\documentclass[11pt,oneside,reqno]{amsart}
\usepackage{graphicx}
\usepackage{amssymb}
\usepackage{amsmath}
\usepackage{amsthm}
\usepackage{amscd}
\usepackage{xcolor}
\usepackage{enumerate}
\usepackage{array}
\usepackage{caption}
\usepackage{subcaption}
\usepackage{siunitx}
\usepackage{textcomp}
\usepackage[T1]{fontenc}

\RequirePackage{dsfont} 
 \scrollmode
\allowdisplaybreaks

\newcolumntype{L}[1]{>{\raggedright\let\newline\\\arraybackslash\hspace{0pt}}m{#1}}
\newcolumntype{C}[1]{>{\centering\let\newline\\\arraybackslash\hspace{0pt}}m{#1}}
\newcolumntype{R}[1]{>{\raggedleft\let\newline\\\arraybackslash\hspace{0pt}}m{#1}}

\newtheorem{thm}{Theorem}[section]

\newtheorem{lem}[thm]{Lemma}

\newtheorem{rem}[thm]{Remark}
\newtheorem{fact}[thm]{Fact}
\theoremstyle{definition}

\newcommand{\N}{\mathbb{Z}_+}

\newcommand{\R}{\mathbb{R}}

\newcommand{\sgn}{\mathop{\mathrm{sgn}}}

\makeatletter
\newcommand*{\defeq}{\mathrel{\rlap{%
                     \raisebox{0.3ex}{$\m@th\cdot$}}%
                     \raisebox{-0.3ex}{$\m@th\cdot$}}%
                     =}
\let\c@table\c@figure
\makeatother

\newcommand{\pic}[4]
{
	\begin{figure}[!htbp]
		\begin{center}
			\includegraphics[width=#2\textwidth]{#1}
			\begin{minipage}[c]{0.8\textwidth}
				\begin{center}
					\caption{#3}\label{#4}
				\end{center}
			\end{minipage}
		\end{center}
	\end{figure}
}

\numberwithin{equation}{section}
\title[Approximation of solutions of DDEs]{Approximation of  solutions of DDEs under nonstandard assumptions via Euler scheme}
\author[Natalia Czy\.z{}ewska]{Natalia Czy\.z{}ewska}
\author[Pawe\l \ M. Morkisz]{Pawe\l \ M. Morkisz}
\author[Pawe\l \ Przyby\l owicz]{Pawe\l \ Przyby\l owicz}
\address[Czy\.z{}ewska]{AGH University of Science and Technology, AGH Doctoral School, Faculty of Applied Mathematics, al. Mickiewicza 30, 30-059 Krak\'ow, Poland}
\email{nczyzew@agh.edu.pl, corresponding author}
\address[Morkisz]{AGH University of Science and Technology, Faculty of Applied Mathematics, al. Mickiewicza 30, 30-059 Krak\'ow, Poland}
\email{morkiszp@agh.edu.pl}
\address[Przyby\l owicz]{AGH University of Science and Technology, Faculty of Applied Mathematics, al. Mickiewicza 30, 30-059 Krak\'ow, Poland}
\email{pprzybyl@agh.edu.pl}
\begin{document}
\begin{abstract}
We deal with approximation of solutions of delay differential equations (DDEs) via the classical Euler algorithm. We investigate the pointwise error of the Euler scheme under nonstandard assumptions imposed on the right-hand side function $f$. Namely, we assume that $f$ is globally of at most linear growth, satisfies globally one-side Lipschitz condition but it is only locally H\"older continuous.  We provide a detailed error analysis of the Euler algorithm under such nonstandard regularity conditions.  Moreover, we report results of numerical experiments.
\end{abstract}
\maketitle	
\section{Introduction}
In this paper we deal with the problem of approximation of solutions  $z: [0, +\infty) \to \R^d$ of delay differential equations (DDEs) of the following form
\begin{eqnarray}
\label{dde_gen}
	\left\{ \begin{array}{ll}
	z'(t)=f(t,z(t),z(t-\tau)), \quad &t\in [0,(n+1)\tau],\\
	z(t)=\eta, &t\in [-\tau,0],
	\end{array} \right.
\end{eqnarray}
with a constant time lag $\tau\in (0,+\infty)$. Here $\eta\in\R^d$, $n\in\N$ is a (finite and fixed) {\it horizon parameter} and a right-hand side function $f: [0,+\infty)\times\R^d\times\R^d\to\R^d$ that satisfies suitable regularity conditions. We investigate the error of the Euler scheme in the case when the right-hand side function is not necessarily globally Lipschitz continuous.

There are two main streams of common used methods for approximation of solutions of delay differential equations: connected with the theory of solving ordinary differential equations (ODE) and the theory of solving partial differential equations (PDE). The former group benefits e.g. from Runge-Kutta methods \cite[Chapter 9]{Balachandran} and \cite{Baker-RDE, Bellen, Torelli} (also special cases as Euler, midpoint or Heun methods), continuous Runge-Kutta methods \cite{Baker-Issues, Bellen}, linear multistep methods \cite{Baker-Issues, Hu} (also Adams–Bashforth methods) or predictor–corrector methods, cf. \cite{Bellen}. The latter takes advantages from reformulating DDE into PDE and for instance uses the infinitesimal generator approach \cite{Breda}, an optimal asymptotic homotopy method \cite{Baleanu} or method of lines \cite{Bellen-Maset, Bellen, Koto}. All of the above and classical literature for solving DDEs \cite{Diekmann, Driver, Smith} assumes mainly global Lipschitz regularity of the right-hand side function $f$ of the underlying DDE \eqref{dde_gen}, alike to classical literature for ordinary differential equations \cite{Butcher, Hairer, Driver}. 

Meanwhile, it turns out that real world applications need nonstandard assumptions. For example, a phase change of metallic materials can be described by a delay differential equation due to delay in the response to the change in processing conditions \cite{NC20, Estrin, Mecking, NC19, Pietrzyk, Szeliga}. A novel case with local Lipschitz condition and uniform boundedness of the right-hand side function was studied in \cite{Baker-RDE}.
In \cite{NC20, NC19} authors weakened this assumptions by considering a case when a right-hand side function is one-dimensional, locally H\"older continuous and monotone. In this paper we generalize the techniques used in \cite{NC20} and study the error of the classical Euler scheme for a multidimensional case, when a right-hand side function is locally H\"older continuous and satisfies one-side Lipschitz conditions. According to our best knowledge till now there were no results in the literature on error analysis for the Euler scheme under such nonstandard assumptions. Our paper can be viewed as a first step into that direction. Moreover, we believe that our approach can be adopted for other  numerical schemes, especially those of higher order.

The main contributions of the paper are as follows.
\begin{itemize}
    \item [(i)] We provide detailed and rigorous theoretical analysis of the error of the Euler scheme under nonstandard assumptions on the right-hand side function $f$. In particular, we show  dependence of the error of the Euler scheme on  H\"older exponents of the right-hand side function $f$ (see Theorem \ref{rate_of_conv_expl_Eul}).
    \item [(ii)] We perform many numerical experiments that confirm our theoretical conclusions.
\end{itemize}
The paper is organized as follows. The statement of the problem, basic notions and definitions are given in Section 2. In this section we also recall the definition of the Euler scheme. Auxiliary analytical properties of DDE \eqref{dde_gen} (such as existence, uniqueness of the solution) under nonstandard assumptions are settled in Section 3. Proof of the main result, that states the upper bound on the error of the Euler algorithm, is given in Section 4. Section 5 contains results of theoretical experiments performed for three exemplary DDEs. In Appendix (Section 6) we gathered and proved some auxiliary results concerning properties of solutions of ODEs in the case when the right-hand side function is of at most linear growth satisfies one-side Lipschitz condition, and is only locally H\"older continuous (Lemma~\ref{odes_exist_sol}). We also show the error estimate for the Euler scheme when applied to ODEs under such regularity conditions (Lemma~\ref{eul_odes}). These results are used in the proof of Theorem~\ref{rate_of_conv_expl_Eul}.
\section{Problem formulation}\label{error}
We denote $x\wedge y=\min\{x,y\}$, $x\vee y=\max\{x,y\}$, and $x_+=\max\{0,x\}$ for all $x,y\in\mathbb{R}$. For $x,y\in\mathbb{R}^d$ we take $\displaystyle{\langle x,y\rangle=\sum\limits_{k=1}^dx_ky_k}$ and  $\|x\|=\langle x,x\rangle^{1/2}$. For the right-hand side function $f: [0,+\infty)\times\R^d\times\R^d\to\R^d$ in the equation (\ref{dde_gen}) we impose the following assumptions:
\begin{itemize}
	\item [(F1)] $f\in C([0,+\infty)\times\R^d\times\R^d;\R^d)$.
	\item [(F2)] There exists a constant $K\in (0,+\infty)$ such that for all $(t,y,z)\in [0,+\infty)\times\R^d\times\R^d$
			\begin{displaymath}
				\| f(t,y,z)\|\leq K(1+\|y\|)(1+\|z\|).
			\end{displaymath}
	\item [(F3)] There exists a constant $H \in \R$ such that for all $(t,z)\in [0,+\infty)\times \R^d$, $y_1,y_2\in\R^d$
			\begin{displaymath}
				\langle y_1-y_2, f(t,y_1,z)-f(t,y_2,z)\rangle \leq H \cdot (1+\|z\|)\cdot \| y_1 - y_2 \|^2.
			\end{displaymath}
	\item [(F4)] There exist $L\in (0,+\infty)$, $\alpha,\beta_1, \beta_2,\gamma\in (0,1]$ such that for all  $t_1,t_2\in [0,+\infty)$, $y_1,y_2,z_1,z_2\in\R^d$
			\begin{eqnarray*}
				\|f(t_1,y_1,z_1)-f(t_2,y_2,z_2)\|&\leq& L\Bigl((1+\|y_1\|+\|y_2\|)\cdot(1+\|z_1\|+\|z_2\|)\cdot |t_1-t_2|^{\alpha}\notag\\
&&+(1+\|z_1\|+\|z_2\|)\|y_1-y_2\|^{\beta_1}\notag\\
&&+(1+\|z_1\|+\|z_2\|)\|y_1-y_2\|^{\beta_2}\notag\\
&&+(1+\|y_1\|+\|y_2\|)\|z_1-z_2\|^\gamma\Bigr).
			\end{eqnarray*}
\end{itemize}
The assumptions above, especially (F3),(F4), are inspired by the real-life model describing evolution of dislocation density, see \cite{NC20, Estrin, Mecking, NC19, Pietrzyk, Szeliga} for detailed description and discussion of the DDE involved. In Lemma \ref{prop_sol_z} below we prove that under the assumptions (F1)-(F3) the equation \eqref{dde_gen} has unique solution $z=z(t)$ on the whole interval $[-\tau,(n+1)\tau]$.

After \cite{Bellen} we recall the definition of the Euler scheme for DDEs of the form \eqref{dde_gen}. For the fixed horizon parameter $n\in\N$ the \textit{Euler scheme} that approximates a solution $z=z(t)$ of \eqref{dde_gen} for $t\in [0,(n+1)\tau]$ is defined recursively for subsequent intervals in the following way (\cite{Bellen}). This iterative way of solving DDEs is often called a method of steps (see \cite{Smith}). We fix the {\it discretization parameter} $N\in\N$ and set
\begin{displaymath}
	t_k^j=j\tau+kh, \quad k=0,1,\ldots,N, \ j=0,1,\ldots,n,
\end{displaymath} 
 where
\begin{equation}
	h=\frac{\tau}{N}.
\end{equation}
Note that for each $j$ the sequence $\{t^j_k\}_{k=0}^N$ provides uniform discretization of the subinterval $[j\tau,(j+1)\tau]$. Discrete approximation of $z$ in $[0,\tau]$ is defined by
\begin{eqnarray}
	y_0^0&=&\eta,\label{eq:43}\\
	y_{k+1}^0&=&y_k^0+h\cdot f(t_k^0,y_k^0,\eta), \quad k=0,1,\ldots, N-1.\label{eq:44}
\end{eqnarray}
Let us assume that the approximations $y_k^{j-1}\approx z(t_k^{j-1})$, $k=0,1,\ldots,N$, have already been defined in  $[(j-1)\tau,j\tau]$. (Notice that for $j=1$ it was done in \eqref{eq:43} and \eqref{eq:44}.) Then for $j=2,3,\ldots,n$ we take
\begin{eqnarray}
\label{expl_euler_1}
	y_0^j&=&y^{j-1}_N,\\
\label{expl_euler_11}	
	y_{k+1}^j&=&y_k^j+h\cdot f(t_k^j,y_k^j,y_k^{j-1}), \quad k=0,1,\ldots, N-1.
\end{eqnarray}
So as the output we obtain the sequence $\{y_k^j\}_{k=0,1,\ldots,N}, j=0,1,\ldots,n$ that provides a discrete approximation of the values $\{z(t_k^{j})\}_{k=0,1,\ldots,N}, j=0,1,\ldots,n$.

The aim is to investigate the error of Euler scheme under the (mild  and nonstandard) assumptions (F1)-(F4), i.e.: upper bound on the following quantity
\begin{equation}
\label{err}
    \max\limits_{0\leq j\leq n}\max\limits_{0\leq k\leq N}\|z(t_k^{j})-y_k^j\|.
\end{equation}
Unless otherwise stated, all constants appearing in the estimates will only depend on $\tau$, $\eta$, $n$, $d$, $K,H,L,\alpha,\beta_1,\beta_2$, and $\gamma$, but not on the discretization parameter $N$. Moreover, we use the same symbol to denote different constants.

\begin{rem}
For a time-dependent delay, i.e. $\tau = \tau(t)$ or a state-dependent delay, i.e. $\tau = \tau(t, z(t))$ one can notice that it is not possible to use methods of steps. When we deal with a variable delay it is highly unlikely to hit the approximation mesh $\{ t_k^j \}_{0\leq k \leq N}^{0\leq j \leq n}$. Because of that, it is required to use an ODE solver with additional dense output (called continuous extension, cf. \cite[Chapter 9]{Balachandran}). However, the direct imitation of the method of steps with a vanishing lag case (a DDE is said to have a vanishing lag at the point $t_*$ if $\tau(t, z(t)) \to 0$ as $t \to t_*$) encounters a natural barrier beyond which the solution cannot proceed \cite{Baker-Issues}.
\end{rem}

\section{Analytical properties of solutions of DDEs}

In the following lemma we show, by using results from the Appendix, that the delay differential equation \eqref{dde_gen} has a unique solution under the assumptions (F1)-(F3). Note that the assumptions are weaker than those known from the standard literature. Namely, we use only one-side Lipschitz assumption and local H\"older condition for the right-hand side function $f$ instead of a global Lipschitz continuity.

Let us introduce following notation. By $\phi_j=\phi_j(t)$ we denote the solution $z$ of \eqref{dde_gen} on the interval $t \in [j\tau,(j+1)\tau]$ for $j \in \N \cup \{0\}$. Let $\phi_{-1}(t):=\eta$ for all $t\in [-\tau,0]$.

\begin{lem} 
\label{prop_sol_z}
Let $\eta\in\mathbb{R}^d$ and let $f$ satisfy (F1)-(F3). Moreover, fix $\tau\in (0,+\infty)$ and $n~\in~\N~\cup~\{0\}$.  Then the equation \eqref{dde_gen} has a unique solution
\begin{equation}
    z\in C^1([0,(n+1)\tau];\mathbb{R}^d).
\end{equation}
Moreover, then there exist $K_0,K_1,\ldots,K_n\geq 0$ such that for $j=0,1,\ldots,n$ 
	\begin{equation}
		\sup\limits_{j\tau\leq t\leq (j+1)\tau}\|\phi_j(t)\|\leq K_j,
	\end{equation}
	and, for all $t,s\in [j\tau,(j+1)\tau]$ 
	\begin{equation}
		\|\phi_j(t)-\phi_j(s)\|\leq \bar K_j |t-s|,
	\end{equation}
	with $\bar K_j=K(1+K_{j-1})(1+K_{j}),$ where $K_{-1}:=\|\eta\|$. 
\end{lem}
\begin{proof} We proceed by induction with respect to $j$.

For $j=0$, the equation \eqref{dde_gen} can be written as
\begin{equation}
\label{dde_gen0}
	z'(t)=f(t,z(t),\eta), \quad t\in [0,\tau],
\end{equation}
with the initial condition $z(0)=\eta$. Denoting by 
\begin{equation}
	g_0(t,y):=f(t,y,\eta), \quad t\in [0,\tau], y\in\R^d,
\end{equation}
we get, by the properties of $f$ namely (F1) and (F2), that $g_0\in C([0,\tau]\times\R^d)$, 
\begin{equation}
	\|g_0(t,y)\|\leq\hat K_0(1+\|y\|),
\end{equation}
with $\hat K_0=K(1+\|\eta\|)$, and by (F3) for all $y_1,y_2\in\R^d$ and $t\in [0,\tau]$
\begin{equation}
    \langle y_1-y_2, g_0(t,y_1)-g_0(t,y_2)\rangle \leq H(1+\|\eta\|)\cdot\| y_1 - y_2 \|^2\leq  \hat H_0\cdot\| y_1 - y_2 \|^2,
\end{equation}
where $\hat H_0=H_+ (1+\|\eta\|)$. Therefore, by Lemma \ref{odes_exist_sol} 
we get that there exists a unique continuously differentiable solution $\phi_0:[0,\tau]\to\R^d$ of the equation \eqref{dde_gen0}, such that
\begin{displaymath}
	\sup\limits_{t\in [0,\tau]}\|\phi_0(t)\|\leq K_0,
\end{displaymath}
where 
$$
K_0=(\|\eta\|+\hat K_0 \tau)e^{\hat K_0 \tau}=(\|\eta\|+K(1+\|\eta\|)\tau)e^{K(1+\|\eta\|) \tau}\geq 0,
$$ 
and for all $t,s\in [0,\tau]$
\begin{displaymath}
	\|\phi_0(t)-\phi_0(s)\|\leq \bar K_0 |t-s|,
\end{displaymath}
where 
$$
\bar K_0=\hat K_0(1+K_0)=K(1+K_{-1})(1+K_0)
$$
and we set $K_{-1}:=\|\eta\|$. In that way it depends only on values of $\|\eta\|,K,\tau$.

Let us now assume that there exists $0\leq j\leq n-1$ such that the statement of the lemma holds for the solution $\phi_j:[j\tau,(j+1)\tau]\to\R^d$. Consider the equation
\begin{equation}
\label{dde_1np1}
	z'(t)=f(t,z(t),\phi_j(t-\tau)), \quad t\in [(j+1)\tau,(j+2)\tau],
\end{equation}
with the initial condition $z((j+1)\tau)=\phi_j((j+1)\tau)$. Let
\begin{equation}
	g_{j+1}(t,y):=f(t,y,\phi_j(t-\tau)), \quad t\in [(j+1)\tau,(j+2)\tau], \ y\in\R^d.
\end{equation}
We get by the induction assumption and from the properties of $f$ that \linebreak $g_{j+1}\in C([(j+1)\tau,(j+2)\tau]\times\R^d;\R^d)$, for all $y\in\R^d$ we have
\begin{equation}
	\|g_{j+1}(t,y)\|\leq K\Bigl(1+\sup\limits_{j\tau\leq t\leq (j+1)\tau}\|\phi_j(t)\|\Bigr)(1+\|y\|)\leq \hat K_{j+1}(1+\|y\|),
\end{equation}
with $\hat K_{j+1}=K(1+K_j)$, and for all $t\in [(j+1)\tau,(j+2)\tau]$, $y_1,y_2\in\R^d$
\begin{eqnarray}
    &&\langle y_1-y_2, g_{j+1}(t,y_1)-g_{j+1}(t,y_2)\rangle \leq H (1+\|\phi_j(t-\tau)\|) \cdot \| y_1 - y_2 \|^2\notag\\
    &&\leq H_+ (1+\|\phi_j(t-\tau)\|) \cdot \| y_1 - y_2 \|^2\leq H_+ \Bigl(1+\sup\limits_{j\tau\leq t\leq (j+1)\tau}\|\phi_j(t)\|\Bigr) \cdot \| y_1 - y_2 \|^2\notag\\
    &&\leq \hat H_{j+1}\cdot \| y_1 - y_2 \|^2,
\end{eqnarray}
where $\hat H_{j+1}=H_+(1+K_j)$. Hence, by Lemma \ref{odes_exist_sol} we get that there exists a unique continuously differentiable solution $\phi_{j+1}:[(j+1)\tau,(j+2)\tau]\to\R^d$ of the equation \eqref{dde_1np1}, such that
\begin{displaymath}
	\sup\limits_{t\in [(j+1)\tau,(j+2)\tau]}\|\phi_{j+1}(t)\|\leq K_{j+1},
\end{displaymath}
where 
$$
K_{j+1}=(K_j+\hat K_{j+1} \tau)e^{\hat K_{j+1} \tau}=(K_j+K(1+K_j)\tau)e^{K(1+K_j)\tau}\geq 0,
$$ 
and for all $t,s\in [(j+1)\tau,(j+2)\tau]$ we have
\begin{displaymath}
	\|\phi_{j+1}(t)-\phi_{j+1}(s)\|\leq \bar K_{j+1} |t-s|,
\end{displaymath}
where $\bar K_{j+1}=\hat K_{j+1}(1+K_{j+1})=K(1+K_j)(1+K_{j+1})$. 

From the above inductive construction we see that the solution of \eqref{dde_gen} is continuous. Moreover, due to the continuity of $f$, $\phi_j$, and $\phi_{j-1}$ we get for any $0\leq j\leq n-1$ that
\begin{eqnarray}
&&\lim\limits_{t\to (j+1)\tau-}z'(t)=\lim\limits_{t\to (j+1)\tau-} \phi'_{j}(t)=\lim\limits_{t\to (j+1)\tau-}f(t,\phi_j(t),\phi_{j-1}(t-\tau))\notag\\
&&=f((j+1)\tau,\phi_j((j+1)\tau),\phi_{j-1}(j\tau))=f((j+1)\tau,\phi_{j+1}((j+1)\tau),\phi_{j}(j\tau))\notag\\
&=&\lim\limits_{t\to (j+1)\tau+}f(t,\phi_{j+1}(t),\phi_{j}(t-\tau))=\lim\limits_{t\to (j+1)\tau+}\phi'_{j+1}(t)=\lim\limits_{t\to (j+1)\tau+}z'(t).\notag
\end{eqnarray}
Hence, the solution of \eqref{dde_gen} is continuously differentiable and the proof is completed.
\end{proof}
\begin{lem} 
\label{props_gn}
Let $\eta\in\mathbb{R}^d$ and let $f$ satisfy (F1)-(F4). Additionally, fix $\tau\in (0,+\infty)$ and $n\in\N \cup \{0\}$. For $j=0,1,\ldots,n$ consider the functions $g_j:[j\tau,(j+1)\tau]\times\R^d\to\R^d$ defined by
\begin{equation}
	g_j(t,y)=f(t,y,\phi_{j-1}(t-\tau)),
\end{equation}
where $\phi_{-1}(t):=\eta$ for $t\in [-\tau,0]$. Then the following holds:
\begin{itemize}
	\item [(i)] $g_j\in C([j\tau,(j+1)\tau]\times\R^d; \R^d)$, $j=0,1,\ldots,n$.
	\item [(ii)] There exist $\hat K_0,\hat K_1\ldots,\hat K_n\in[0,+\infty)$ such that for all $j=0,1,\ldots,n$, \linebreak $(t,y)\in [j\tau,(j+1)\tau]\times\R^d$
		\begin{displaymath}
			\|g_j(t,y)\|\leq\hat K_{j}(1+\|y\|).
		\end{displaymath}
	\item [(iii)]  There exist $\hat H_0,\hat H_1\ldots,\hat H_n\in[0,+\infty)$ such that for all $j=0,1,\ldots,n$, \linebreak $t\in [j\tau,(j+1)\tau], y_1,y_2\in\R^d$
	\begin{displaymath}
		\langle y_1-y_2, g_j(t,y_1)-g_j(t,y_2)\rangle \leq \hat H_j\cdot \| y_1 - y_2 \|^2.
	\end{displaymath}
	\item [(iv)] There exist $\hat L_0,\hat L_1\ldots,\hat L_n\in[0,+\infty)$ such that for all $j=0,1,\ldots,n$, \linebreak $t_1,t_2\in [j\tau,(j+1)\tau], y_1,y_2\in\R^d$
	\begin{displaymath}
		\|g_j(t_1,y_1)-g_j(t_2,y_2)\|\leq \hat L_j \Bigl((1+\|y_1\|+\|y_2\|)\cdot |t_1-t_2|^{\alpha\wedge\gamma}+\|y_1-y_2\|^{\beta_1}+\|y_1-y_2\|^{\beta_2}\Bigr).
	\end{displaymath}
\end{itemize}
\end{lem}

\begin{proof} Conditions (i), (ii), and (iii) follow from the proof of Lemma \ref{prop_sol_z}. By the assumption (F4) and Lemma \ref{prop_sol_z} we get for all $t_1,t_2\in [j\tau,(j+1)\tau]$, $y_1,y_2\in\R^d$ that
\begin{eqnarray*}		
	&&\|g_j(t_1,y_1)-g_j(t_2,y_2)\|=\|f(t_1,y_1,\phi_{j-1}(t_1-\tau))-f(t_2,y_2,\phi_{j-1}(t_2-\tau))\|\\
	&&\leq L\Bigl( (1+2K_{j-1}) (1+\|y_1\|+\|y_2\|) |t_1-t_2|^{\alpha}\notag\\
	&&+(1+2K_{j-1})\Bigl(\|y_1-y_2\|^{\beta_1}+\|y_1-y_2\|^{\beta_2}\Bigr)\\
	&&+\bar K_{j-1}^{\gamma}\Bigl(1+\|y_1\|+\|y_2\|\Bigr)|t_1-t_2|^{\gamma}\biggr)\\
	&&\leq \hat L_j\Bigl( (1+\|y_1\|+\|y_2\|) \cdot |t_1-t_2|^{\alpha\wedge\gamma} + \|y_1-y_2\|^{\beta_1}+\|y_1-y_2\|^{\beta_2}\Bigr),
	\end{eqnarray*}
	where $\hat L_j=L\max\{1+2K_{j-1},\bar K_{j-1}^\gamma\}\cdot 2(1+\tau)$, $\hat L_0=L(1+2\|\eta\|)\cdot 2 (1+\tau)$, and \linebreak $\bar K_{-1}:=0, K_{-1}:=\|\eta\|$.
\end{proof}

\section{Error of the Euler scheme for delay differential equations}
In this section we provide a proof of the main result that consists of the upper bound on the error \eqref{err} for the Euler algorithm. In the proof we shall use the following lemma.
\begin{lem}
    \label{euler_bound_lem}
    Let $\tau\in (0,+\infty)$, $\eta\in\mathbb{R}^d$ and let $f$ satisfy (F1)-(F4). For any $n\in\N\cup\{0\}$ there exist $\tilde K_0,\ldots\tilde K_n\in (0,+\infty)$, $\tilde K_j=\tilde{K}_j(K,\eta,\tau)$, such that for all $N\in\N$
    \begin{equation}
        \max\limits_{0\leq j\leq n}\max\limits_{0\leq k\leq N}\|y_k^j\|\leq\max\limits_{0\leq j\leq n}\tilde K_j.
    \end{equation}
\end{lem}
\begin{proof}
We proceed by induction. Note that
\begin{equation}
    y^0_{k+1}=y^0_k+h\cdot g_0(t^0_k,y^0_k), \quad k=0,1,\ldots,N-1,
\end{equation}
where $y^0_0=\eta$ and $g_0(t,y)=f(t,y,\eta)$, so  by Lemmas \ref{props_gn},  \ref{eul_odes}  we get that
\begin{equation}
    \max\limits_{0\leq k\leq N}\|y_k^0\|\leq \tilde K_0,
\end{equation}
where $\tilde K_0=\tilde K_0(K,\eta,\tau)$. Now, let us assume that there exist $j=0,1,\ldots,n-1$ and \linebreak $\tilde K_j=\tilde K_j(K,\eta,\tau)\in (0,+\infty)$ such that for all $N\in\N$
\begin{equation}
    \max\limits_{0\leq k\leq N}\|y_k^j\|\leq \tilde{K}_j,
\end{equation}
which is obviously satisfied for $j=0$. By \eqref{expl_euler_11} and (F2) assumption we get for \linebreak $k=0,1,\ldots,N-1$ that
\begin{equation}
    \|y^{j+1}_{k+1}\|\leq\|y^{j+1}_k\|+h\|f(t_k^{j+1},y_k^{j+1},y_k^j)\|\leq (1+h\tilde C_{j+1})\|y^{j+1}_k\|+h\tilde C_{j+1},
\end{equation}
where $\|y_0^{j+1}\|=\|y_N^j\|\leq\tilde K_j$ and $\tilde C_{j+1}=K(1+\tilde K_j)$. From the discrete version of Gronwall's lemma we obtain
\begin{equation}
\label{gronwall_k_tilde}
   \max\limits_{0\leq k\leq N} \|y_k^{j+1}\|\leq\tilde K_{j+1},
\end{equation}
with $\displaystyle{\tilde K_{j+1}=e^{\tau\tilde C_{j+1}}(\tilde K_j+1)-1}$ and $\tilde K_{j+1}=\tilde K_{j+1}(K,\eta,\tau)$. This completes the proof.
\end{proof}
Below we stated and prove  the main result of the paper.
\begin{thm} 
\label{rate_of_conv_expl_Eul} 
Let $\tau\in (0,+\infty)$, $\eta\in\R^d$ and let $f$ satisfy (F1)-(F4). For any $n\in\N\cup\{0\}$ there exist $C_0,C_1,\ldots,C_n\geq 0$ such that for  $N\geq 2\lceil \tau\rceil$ the following holds
	\begin{equation}
	\label{thm_eq1}
		\max\limits_{0\leq k\leq N}\|\phi_0(t_k^0)-y_k^0\|\leq C_0 (h^{\alpha\wedge\gamma}+h^{\beta_1}+h^{\beta_2}),
	\end{equation}
	and for $j=1,2,\ldots,n$
	\begin{equation}
		\max\limits_{0\leq k\leq N}\|\phi_j(t_k^j)-y_k^j\|\leq C_j\sum\limits_{l=1}^j \Bigl(h^{\frac{1}{2}\gamma^{l-1}}+h^{\gamma^l \cdot (\alpha\wedge\gamma)}+h^{\beta_1\cdot\gamma^l}+h^{\beta_2\cdot\gamma^l}\Bigr),
	\end{equation}
	where $\phi_j=\phi_j(t)$ is the solution of \eqref{dde_gen} on the interval $[j\tau,(j+1)\tau]$. In particular, if $\gamma=1$ then for $j=1,2,\ldots,n$
	\begin{equation}
		\max\limits_{0\leq k\leq N}\|\phi_j(t_k^j)-y_k^j\|\leq jC_j (h^{1/2}+h^{\alpha}+h^{\beta_1}+h^{\beta_2}) . 
	\end{equation}
\end{thm}
\begin{proof}
For $t\in [0,\tau]$ we approximate the solution $z$ of (\ref{dde_gen}) by the  Euler method 
\begin{eqnarray}
	y_0^0&=&\eta,\\
	y_{k+1}^0&=&y_k^0+h\cdot g_0(t_k^0,y_k^0), \quad k=0,1,\ldots, N-1,
\end{eqnarray}
where $g_0(t,y)=f(t,y,\eta)$. Applying Lemmas \ref{eul_odes}, \ref{props_gn} to $\xi:=\eta$, $g:=g_0$, $[a,b]:=[0,\tau]$, $\Delta:=0$ we get that
\begin{equation}
\label{est_euler_0tcr}
		\max\limits_{0\leq k\leq N}\|\phi_0(t_k^0)-y_k^0\|\leq \tilde C_2 (1+\|\eta\|) (h^{\alpha \wedge \gamma}+h^{\beta_1}+h^{\beta_2}).
\end{equation}
Therefore \eqref{thm_eq1} is proved.

Starting from the interval $[\tau,2\tau]$ we proceed by induction. Namely, for $j=1$ and $t\in [\tau,2\tau]$ the DDE \eqref{dde_gen} reduces to the following ODE
\begin{equation}
\label{dde_23}
	z'(t)=g_1(t,z(t)), \quad t\in [\tau,2\tau],
\end{equation}
with the initial value $z(\tau)=\phi_0(\tau)=\phi_0(t^0_N)$ and $g_1(t,y)=f(t,y,\phi_0(t-\tau))$. We approximate the solution $z$ of \eqref{dde_23} by the auxiliary Euler scheme 
\begin{eqnarray}
\label{aux_expl_euler_10}
	\tilde y_0^1&=&y^{1}_0=y_N^0,\\
\label{aux_expl_euler_1k}	
	\tilde y_{k+1}^1&=&\tilde y_k^1+h\cdot g_1(t_k^1,\tilde y_k^1), \quad k=0,1,\ldots, N-1,
\end{eqnarray}
and from \eqref{est_euler_0tcr} we get
\begin{equation}
	\|z(\tau)-\tilde y_0^1\|=\|\phi_1(t_0^1)-\tilde y_0^1\|=\|\phi_0(t_N^0)-y_N^0\|\leq \tilde C_2 (1+\|\eta\|) (h^{\alpha \wedge \gamma}+h^{\beta_1}+h^{\beta_2}).
\end{equation}
Applying Lemmas \ref{props_gn}, \ref{prop_sol_z} and \ref{eul_odes} for $\xi:=\phi_0(t^0_N)$, $g:=g_1$, $[a,b]:=[\tau,2\tau]$, \linebreak $\Delta:=\tilde C_2 (1+\|\eta\|) (h^{\alpha \wedge \gamma}+h^{\beta_1}+h^{\beta_2})$ we get that
\begin{equation}
	\|\phi_0(t^0_N)\|\leq K_0,
\end{equation}
and 
\begin{equation}
\label{est_phi1_tyk1}
		\max\limits_{0\leq k\leq N}\|\phi_1(t_k^1)-\tilde y_k^1\|\leq  C_0  (h^{\alpha \wedge \gamma}+h^{\beta_1}+h^{\beta_2}).
\end{equation}
Therefore, we have for $k=0,1,\ldots,N$
\begin{equation}
\label{est_phi1y1k}
        \|\phi_1(t_k^1)-y_k^1\| \leq \|\phi_1(t_k^1)-\tilde y_k^1\| + \|\tilde y_k^1 - y_k^1\|\leq C_0  (h^{\alpha \wedge \gamma}+h^{\beta_1}+h^{\beta_2})+\|\tilde y_k^1 - y_k^1\|
\end{equation}
and we need to estimate $\|\tilde y_k^1 - y_k^1\|$. Let us denote by
\begin{equation}
\label{local_err_1}
	e_k^1:=\tilde y_k^1-y_k^1, \quad k=0,1,\ldots,N,
\end{equation}
where, by \eqref{aux_expl_euler_10},  $e_0^1=\tilde y_0^1-y_0^1=0$. From \eqref{aux_expl_euler_1k} and  \eqref{expl_euler_11} we have for $k~=~0,1,\ldots,N-1$ that
\begin{equation}
\label{rek_ek1}
	e_{k+1}^1=e_k^1+h\mathcal{R}^1_k+h\mathcal{L}^1_k,
\end{equation}
where
\begin{equation}
\label{def_rk1}
	\mathcal{R}^1_k=f(t_k^1,\tilde y_k^1,\phi_0(t_k^0))-f(t_k^1,y_k^1,\phi_0(t_k^0)), 
\end{equation}
and
\begin{equation}
\label{def_lk1}
	\mathcal{L}^1_k=f(t_k^1,y_k^1,\phi_0(t_k^0))-f(t_k^1,y_k^1,y_k^0).
\end{equation}
From \eqref{rek_ek1} we obtain that
\begin{equation}
\label{err_norm_eq}
	\|e_{k+1}^1-h\mathcal{L}^1_k\|^2=\|e_k^1+h\mathcal{R}^1_k\|^2,
\end{equation}
where
\begin{displaymath}
	\|e_{k+1}^1-h\mathcal{L}_k^1\|^2 = \|e_{k+1}^1\|^2 - 2h\langle e_{k+1}^1, \mathcal{L}_k^1 \rangle + h^2\|\mathcal{L}_k^1\|^2
\end{displaymath}
\begin{equation}
\label{norm_square_right1}
	\|e_k^1+h\mathcal{R}_k^1\|^2 = \|e_k^1\|^2 + 2h\langle e_k^1, \mathcal{R}_k^1 \rangle + h^2\|\mathcal{R}_k^1\|^2.
\end{equation}
Since $h^2\|\mathcal{L}_k^1\|^2 \geq 0$, we get
\begin{equation}
\label{norm_square_left}
	\|e_{k+1}^1-h\mathcal{L}_k^1\|^2 \geq \|e_{k+1}^1\|^2 - 2h\langle e_{k+1}^1, \mathcal{L}_k^1 \rangle.
\end{equation}
while from the assumption (F3) and from Lemma \ref{prop_sol_z} we get
\begin{eqnarray*}
    &&\langle e_k^1, \mathcal{R}_k^1 \rangle = \left\langle \tilde y_k^1-y_k^1, \ f\left(t_k^1,\tilde y_k^1,\phi_0(t_k^0)\right)-f\left(t_k^1,y_k^1,\phi_0(t_k^0)\right) \right\rangle \\
    &&\leq H \left(1 + \|\phi_0(t_k^0)\|\right) \|e_k^1\|^2 \leq H_+ \left(1 + \|\phi_0(t_k^0)\|\right) \|e_k^1\|^2 \\
    &&\leq H_+ \left(1 + \sup_{0\leq t \leq \tau} \|\phi_0(t)\|\right) \|e_k^1\|^2\leq H_+ (1 + K_0) \|e_k^1\|^2.
\end{eqnarray*}
 The fact above together with \eqref{norm_square_right1} implies
\begin{equation}
\label{norm_square_right2}
	\|e_k^1+h\mathcal{R}_k^1\|^2 \leq \|e_k^1\|^2 + 2hH_+ (1 + K_0) \|e_k^1\|^2 + h^2\|\mathcal{R}_k^1\|^2.
\end{equation}
Hence, by using \eqref{err_norm_eq}, \eqref{norm_square_left}, and \eqref{norm_square_right2} we obtain
\begin{equation}
\label{err_next_step}
	\|e_{k+1}^1\|^2 \leq \left( 1 + 2hH_+ (1 + K_0)\right) \|e_k^1\|^2 + h^2\|\mathcal{R}_k^1\|^2 + 2h\langle e_{k+1}^1, \mathcal{L}_k^1 \rangle.
\end{equation}
Moreover, by the Cauchy-Schwarz inequality
\begin{equation}
\label{err_dot_product}
    \langle e_{k+1}^1, \mathcal{L}_k^1 \rangle \leq |\langle e_{k+1}^1, \mathcal{L}_k^1 \rangle| \leq \frac{1}{2} (\|e_{k+1}^1\|^2 + \|\mathcal{L}_k^1\|^2).
\end{equation}
Hence, combining \eqref{err_next_step} with \eqref{err_dot_product} we have
\begin{equation}
\label{}
	\|e_{k+1}^1\|^2 \leq \left( 1 + \hat C_1 h\right) \|e_k^1\|^2 + h^2\|\mathcal{R}_k^1\|^2 + h\|e_{k+1}^1\|^2 + h\|\mathcal{L}_k^1\|^2
\end{equation}
where $\hat C_1:=2H_+ (1 + K_0)$. For $N \geq 2 \lceil \tau \rceil$ we have that $h\in (0,1/2)$, and by the Fact~\ref{fakt_h} we obtain 
\begin{equation}
\label{err_square}
	 \|e_{k+1}^1\|^2 \leq (1+2h)(1+\hat C_1 h) \|e_k^1\|^2 + 2h^2\|\mathcal{R}_k^1\|^2 + 2h\|\mathcal{L}_k^1\|^2
\end{equation}
for $k=0,1,\dots, N-1$. Recall the well-known facts that for all $\varrho \in (0,1]$ and $x,y\geq 0$ it holds
\begin{displaymath}
 x^\varrho \leq 1 + x.
\end{displaymath}
and
\begin{equation}
\label{inequality_rho}
    (x+y)^\varrho \leq x^\varrho + y^\varrho.
\end{equation}
Thereby,  from the assumption (F4) and by Lemma~\ref{prop_sol_z} we have the following estimate
\begin{eqnarray}
\label{est_Rk1}
	&&\|\mathcal{R}_k^1\| =\| f(t_k^1,\tilde y_k^1,\phi_0(t_k^0))-f(t_k^1,y_k^1,\phi_0(t_k^0)) \|\notag \\
	&&\leq L \left[ \left(1+2 \|\phi_0(t_k^0)\|\right) \|e_k^1\|^{\beta_1} + \left(1+2 \|\phi_0(t_k^0)\|\right) \|e_k^1\|^{\beta_2} \right] \notag\\
	&&= L \left(1+2 \|\phi_0(t_k^0)\|\right) \left(\|e_k^1\|^{\beta_1} + \|e_k^1\|^{\beta_2}\right) \notag\\
	&&\leq 2L \left(1+2 \sup_{0\leq t \leq \tau}\|\phi_0(t)\|\right) (1 + \|e_k^1\|) \leq 2L (1+2K_0)(1+\|e_k^1\|).
\end{eqnarray}
while by Lemma~\ref{euler_bound_lem} and \eqref{est_euler_0tcr}  we obtain
\begin{eqnarray}
\label{est_Lk1}
	&&\|\mathcal{L}_k^1\| = \| f(t_k^1,y_k^1,\phi_0(t_k^0))-f(t_k^1,y_k^1,y_k^0) \| \notag\\
	&&\leq L (1+2 \|y_k^1\|) \|\phi_0(t_k^0) - y_k^0 \|^\gamma
	\leq \bar C_1 (h^{\gamma(\alpha \wedge\gamma)}+h^{\beta_1 \gamma}+h^{\beta_2 \gamma}),
\end{eqnarray}
where $\bar C_1 := L (1+2 \tilde K_1)\tilde C_2^\gamma (1+\|\eta\|)^\gamma$. Hence, by \eqref{err_square}, \eqref{est_Rk1} and \eqref{est_Lk1}
\begin{equation}
	\|e_{k+1}^1\|^2 \leq (1 + \bar D_1 h)\|e_k^1\|^2 + D_1 h^2 + M_1 h (h^{2\gamma(\alpha \wedge\gamma)}+h^{2\beta_1 \gamma}+h^{2\beta_2 \gamma})
\end{equation}
for all $h \in (0, \frac{1}{2})$, $k=0,1,\dots, N-1$. We stress that $\bar D_1,D_1,M_1$ do not depend on $N$. Since  $\|e_0^1\|^2 = 0$, from the discrete Gronwall's inequality we get for all $k=0,1,\dots, N$
\begin{displaymath}
	\|e_{k}^1\|^2 \leq \left(\overline{\widetilde{C}}_1\right)^2 \left( h + h^{2\gamma(\alpha \wedge\gamma)}+h^{2\beta_1 \gamma}+h^{2\beta_2 \gamma} \right),
\end{displaymath}
where $\left(\overline{\widetilde{C}}_1\right)^2 = \frac{e^{\bar D_1 \tau}-1}{\bar D_1} \cdot \max \{D_1, M_1 \}$. Hence
\begin{equation}
\label{err_after_gronwall}
    \max_{0 \leq k \leq N} \|e_k^1 \| \leq \overline{\widetilde{C}}_1 \left( h^{\frac{1}{2}} + h^{\gamma(\alpha \wedge\gamma)}+h^{\beta_1 \gamma}+h^{\beta_2 \gamma} \right).
\end{equation}
From \eqref{est_phi1y1k} and \eqref{err_after_gronwall} we get
\begin{equation}
    \|\phi_1(t_k^1)-y_k^1\|\leq C_0  (h^{\alpha \wedge \gamma}+h^{\beta_1}+h^{\beta_2})+ \|e_k^1\|\leq  C_1 \left( h^{\frac{1}{2}} + h^{\gamma (\alpha \wedge \gamma)}+h^{\beta_1 \gamma}+h^{\beta_2 \gamma} \right),
\end{equation}
for $k=0,1,\dots, N$ where $C_1$ does not depend on $N$.

Let us now assume that there exist $j \in \{ 1, 2, \dots, n-1 \}$ and $C_j$, which does not depend on $N$, such that for all $k=0,1,\dots,N$ it holds
\begin{equation}
\label{inductive_assumption}
    \|\phi_j(t_k^j)-y_k^j\| \leq C_j \sum_{l=1}^j \left( h^{\frac{1}{2} \gamma^{l-1}} + h^{\gamma^l (\alpha \wedge \gamma)} + h^{\beta_1 \gamma^l} + h^{\beta_2 \gamma^l} \right).
\end{equation}
(For $j=1$  the statement has already been proven.) For $t \in [(j+1)\tau, (j+2)\tau]$ we consider the following ODE
\begin{equation}
\label{dde_1j}
	z'(t)=g_{j+1}(t,z(t)), \quad t\in [(j+1)\tau, (j+2)\tau],
\end{equation}
with the initial value $z((j+1)\tau)=\phi_j((j+1)\tau)=\phi_{j}(t_N^j)=\phi_j(t_0^{j+1})$ and $g_{j+1}(t,y)=f(t,y,\phi_j(t-\tau))$. We approximate \eqref{dde_1j} by the following auxiliary Euler scheme 
\begin{eqnarray}
\label{aux_expl_euler_j0}
	\tilde y_0^{j+1}&=&y_0^{j+1}=y_N^j,\\
\label{aux_expl_euler_jk}
	\tilde y_{k+1}^{j+1}&=&\tilde y_k^{j+1}+h\cdot g_{j+1}(t_k^{j+1},\tilde y_k^{j+1}), \quad k=0,1,\ldots, N-1.
\end{eqnarray}
Hence, from the induction assumption \eqref{inductive_assumption}
\begin{displaymath}
	\|z((j+1)\tau) - \tilde y_0^{j+1}\| = \|\phi_j(t_N^j)- y_N^j\| 
	\leq C_j \sum_{l=1}^j \left( h^{\frac{1}{2} \gamma^{l-1}} + h^{\gamma^l (\alpha \wedge \gamma)} + h^{\beta_1 \gamma^l} + h^{\beta_2 \gamma^l} \right).
\end{displaymath}
Applying Lemmas \ref{prop_sol_z}, \ref{props_gn} and \ref{eul_odes}  for $\xi:=\phi_j(t_N^j)$, $g:=g_{j+1}$, $[a,b]:=[(j+1)\tau,(j+2)\tau]$, \linebreak $\displaystyle{\Delta:=C_j \sum_{l=1}^j \left( h^{\frac{1}{2} \gamma^{l-1}} + h^{\gamma^l (\alpha \wedge \gamma)} + h^{\beta_1 \gamma^l} + h^{\beta_2 \gamma^l} \right)}$, we obtain
\begin{displaymath}
\|\phi_j(t_N^j)\|\leq K_j,
\end{displaymath}
and 
\begin{eqnarray}
\label{est_phiyj}
		&&\max\limits_{0\leq k\leq N}\|\phi_{j+1}(t_k^{j+1})-\tilde y_k^{j+1}\| \leq \tilde C_{2}(1+\|\phi_j(t_N^j)\|) (\Delta + h^{\alpha \wedge \gamma} + h^{\beta_1} + h^{\beta_2})\notag \\
		&&\leq \overline{\widetilde{C}}_{j+1} \left( \sum_{l=1}^j h^{\frac{1}{2} \gamma^{l-1}} + h^{\gamma^l (\alpha \wedge \gamma)} + h^{\beta_1 \gamma^l} + h^{\beta_2 \gamma^l} \right).
\end{eqnarray}
Hence, we have for $k=0,1,\ldots,N$
\begin{eqnarray}
\label{est_phijy1k}
        &&\|\phi_{j+1}(t_k^{j+1})-y_k^{j+1}\| \leq \|\phi_{j+1}(t_k^{j+1})-\tilde y_k^{j+1}\| + \|e_k^{j+1}\|\notag\\
        &&\leq\overline{\widetilde{C}}_{j+1} \left( \sum_{l=1}^j h^{\frac{1}{2} \gamma^{l-1}} + h^{\gamma^l (\alpha \wedge \gamma)} + h^{\beta_1 \gamma^l} + h^{\beta_2 \gamma^l} \right)+\|e_k^{j+1}\|,
\end{eqnarray}
where
\begin{displaymath}
	e_k^{j+1}:=\tilde y_k^{j+1}-y_k^{j+1}, \quad k=0,1,\ldots,N,
\end{displaymath}
and $e_0^{j+1}=\tilde y_0^{j+1}-y_0^{j+1}=0$ by \eqref{aux_expl_euler_j0}. Using analogous arguments as in \eqref{rek_ek1}-\eqref{est_Lk1} we obtain  for $h \in (0, 1/2)$ and $k=0,1,\ldots,N-1$ that
\begin{equation}
\label{rec_ekj}
    \|e_{k+1}^{j+1}\|^2 \leq (1+2h)(1+\hat C_{j+1} h) \|e_k^{j+1}\|^2 + 2h^2\|\mathcal{R}_k^{j+1}\|^2 + 2h \|\mathcal{L}_k^{j+1}\|^2
\end{equation}
where $\hat C_{j+1}$ does not depend on $N$. Moreover
\begin{equation}
\label{est_rkj11}
	\|\mathcal{R}_k^{j+1}\| = \| f(t_k^{j+1},\tilde y_k^{j+1},\phi_j(t_k^j))-f(t_k^{j+1},y_k^{j+1},\phi_j(t_k^j)) \|\leq 2L (1+2K_j)(1+\|e_k^{j+1}\|),
\end{equation}
and, in particular by using \eqref{inductive_assumption} and Lemma \ref{euler_bound_lem}, we get
\begin{eqnarray}
\label{est_lkj11}
	&&\|\mathcal{L}_k^{j+1}\| = \| f(t_k^{j+1},y_k^{j+1},\phi_j(t_k^j))-f(t_k^{j+1},y_k^{j+1},y_k^j) \|\leq L (1+2 \|y_k^{j+1}\|) \|\phi_j(t_k^j) - y_k^j \|^\gamma\notag \\
	&&\leq \bar C_{j+1} \sum_{l=1}^j \left( h^{\frac{1}{2} \gamma^l} + h^{\gamma^{l+1} (\alpha \wedge \gamma)}+h^{\beta_1 \gamma^{l+1}}+h^{\beta_2 \gamma^{l+1}} \right),
\end{eqnarray}
where $\bar C_{j+1} := L (1+2 \tilde K_{j+1}) C_j^\gamma$. Hence, from \eqref{rec_ekj}, \eqref{est_rkj11}, and \eqref{est_lkj11} we have for $h \in (0, \frac{1}{2})$ and $k=0,1,\ldots,N$
\begin{displaymath}
    \|e_{k+1}^{j+1}\|^2 \leq (1+\bar D_{j+1} h) \|e_k^{j+1}\|^2 + D_{j+1} h^2 + M_{j+1} h \left[\sum_{l=1}^j \left( h^{\frac{1}{2}\gamma^l} + h^{\gamma^{l+1} (\alpha \wedge \gamma)} + h^{\beta_1 \gamma^{l+1}} + h^{\beta_2 \gamma^{l+1}}\right) \right]^2 \\
\end{displaymath}
where $D_{j+1}$, $M_{j+1}$, $\bar D_{j+1}$ does not depend on $N$. Since $\|e_0^{j+1}\|^2=0$, by the discrete Gronwall's lemma we get that for all $k=0,1,\dots, N$
\begin{displaymath}
    \|e_k^{j+1}\|^2 \leq \left(\overline{\widetilde{C}}_{j+1}\right)^2  \left\{ h + \left[\sum_{l=1}^j \left( h^{\frac{1}{2}\gamma^l} + h^{\gamma^{l+1} (\alpha \wedge \gamma)} + h^{\beta_1 \gamma^{l+1}} + h^{\beta_2 \gamma^{l+1}}\right) \right]^2 \right\}
\end{displaymath}
where $\displaystyle{\left(\overline{\widetilde{C}}_{j+1}\right)^2 = \frac{e^{\bar D_{j+1} \tau} - 1}{\bar D_{j+1}} \cdot \max \{D_{j+1}, M_{j+1} \}}$. Thus
\begin{eqnarray}
\label{est_ekj11}
    &&\max\limits_{0\leq k \leq N}\|e_k^{j+1}\| \leq \overline{\widetilde{C}}_{j+1} \left[ h^{\frac{1}{2}} + \sum_{l=1}^j \left( h^{\frac{1}{2}\gamma^l} + h^{\gamma^{l+1} (\alpha \wedge \gamma)} + h^{\beta_1 \gamma^{l+1}} + h^{\beta_2 \gamma^{l+1}}\right) \right]\notag\\
    &&\leq 2 \overline{\widetilde{C}}_{j+1}\sum_{l=1}^j \left( h^{\frac{1}{2}\gamma^l} + h^{\gamma^{l+1} (\alpha \wedge \gamma)} + h^{\beta_1 \gamma^{l+1}} + h^{\beta_2 \gamma^{l+1}} \right).
\end{eqnarray}
Therefore, from \eqref{est_phijy1k} and \eqref{est_ekj11} we get for $k=0,1,\dots, N$
\begin{displaymath}
	\|\phi_{j+1}(t_k^{j+1})-y_k^{j+1}\|\leq C_{j+1}\sum_{l=1}^{j+1} \left( h^{\frac{1}{2}\gamma^{l-1}} + h^{\gamma^{l} (\alpha \wedge \gamma)} + h^{\beta_1 \gamma^{l}} + h^{\beta_2 \gamma^{l}} \right),
\end{displaymath}
where $C_{j+1}$ does not depend on $N$. This ends the proof. 
\end{proof}
\begin{rem}\rm  Instead of (F4) we can consider the following general assumption:
\begin{itemize}
		\item [(F4*)] 
		There exist $L\geq 0$, $p, q \in \N$, $\alpha,\beta_i,\gamma_j \in (0,1]$ for $i=1,\ldots, p$, $j=1,\ldots, q$ such that for all $t_1,t_2\in [0,+\infty)$, $y_1,y_2,z_1,z_2\in\R^d$
			\begin{eqnarray*}
				\|f(t_1,y_1,z_1)-f(t_2,y_2,z_2)\|&\leq& L\Bigl((1+\|y_1\|+\|y_2\|)\cdot(1+\|z_1\|+\|z_2\|)\cdot |t_1-t_2|^{\alpha}\notag\\
&&+(1+\|z_1\|+\|z_2\|) \sum_{i=1}^p \|y_1-y_2\|^{\beta_i}\notag\\
&&+(1+\|y_1\|+\|y_2\|) \sum_{j=1}^q \|z_1-z_2\|^{\gamma_j}\Bigr).
			\end{eqnarray*}
\end{itemize}
Because of the notational simplicity we considered only the case when  (F4) is assumed. However, we stress that the general case, when (F4*) is assumed,  can be treated by using the same proof technique as used in the proof of Theorem \ref{rate_of_conv_expl_Eul}.
\end{rem}
\section{Numerical experiments}
We have conducted the numerical experiments with the Python implementations of the method described in \eqref{expl_euler_1} and \eqref{expl_euler_11}. The numerical experiments were carried on using Intel\textsuperscript{\tiny\textregistered} Xeon\textsuperscript{\tiny\textregistered} CPU E5-2650 v4 @ 2.20GHz. The exact solutions for the analyzed problems are not known so we proceeded as follows to estimate the empirical rate of convergence. Theorem~\ref{rate_of_conv_expl_Eul} provides the theoretical convergence rate of the Euler method, hence, the approximation computed by the Euler method on the dense mesh will be used as the referential solution ($1000$ times more points than the densest mesh used, i.e. for $N=9\cdot 10^3 \cdot 2^{13}$ and $N=9\cdot 10^3 \cdot 2^{10}$ for SIR model).

The results are presented in the figures comprising the test results for convergence rate, on the y axis there is $-\log_{10} \textrm{(err)}$ and on the x axis the $\log_{10}N$, where $\textrm{err}$ and $N$ are given by \eqref{err}. Then, the slope of the linear regression corresponds to the empirical convergence rate and is always presented.

The empirical convergence rate for presented examples is close to one. It can indicate that in the considered case lower bounds of the error are not sharp. 
\vspace{10pt}
\subsection{Metal phase change model}
Firstly, we consider a modified model describing a phase change of metallic materials from \cite[Chapter 3.3]{Pietrzyk}, see also \cite{NC20, Estrin, Mecking, NC19, Szeliga}. It can be described by delay differential equation due to delay in the response to the change in processing conditions. We consider the following case  
\begin{equation}
\label{example_metal}
    f(t,y,z)=A-B\cdot  \sgn(y)\cdot |y|-C\cdot \sgn(y)\cdot |y|^{\varrho}\cdot |z|^\gamma +D \cdot y \cdot |z|^\gamma,
\end{equation}
where $\varrho, \gamma \in (0,1]$, $A,B,C > 0$, and $D \in \R$. Note that the function \eqref{example_metal} satisfies assumptions (F1)-(F4), see Fact \ref{metal_assumptions_1}. For the numerical experiment we take following values of parameters $A = 1.7137$, $B = 0.7769$, $C = 0.5895$, $D = -0.82615$, $\varrho = 0.973$, $\gamma = 0.714$, $\tau = 9.2603$, $z_0 = 0.05854$, $t_0 = 0$, and $n=5$.
The approximated solution of \eqref{example_metal} is presented in Figure~\ref{pic:sol_test_eq}. Note that the solution graph is similar to the solutions from \cite{NC20, NC19, Szeliga}. The test results for this case are given in Figure~\ref{pic:conv_rate_test_eq}. The computation time is linear with respect to the number of computation points and for $N \cdot n = 90$ equals $0.0023 s$ and for $N \cdot n = 368640$ equals $5.775 s$.

We also tested equation \eqref{example_metal} with different values of the parameters, exemplary solutions are presented in Figure~\ref{pic:sol_test_eq_5.3_other}. The test results regarding the convergence rate for all examples are always close to one.

\pic{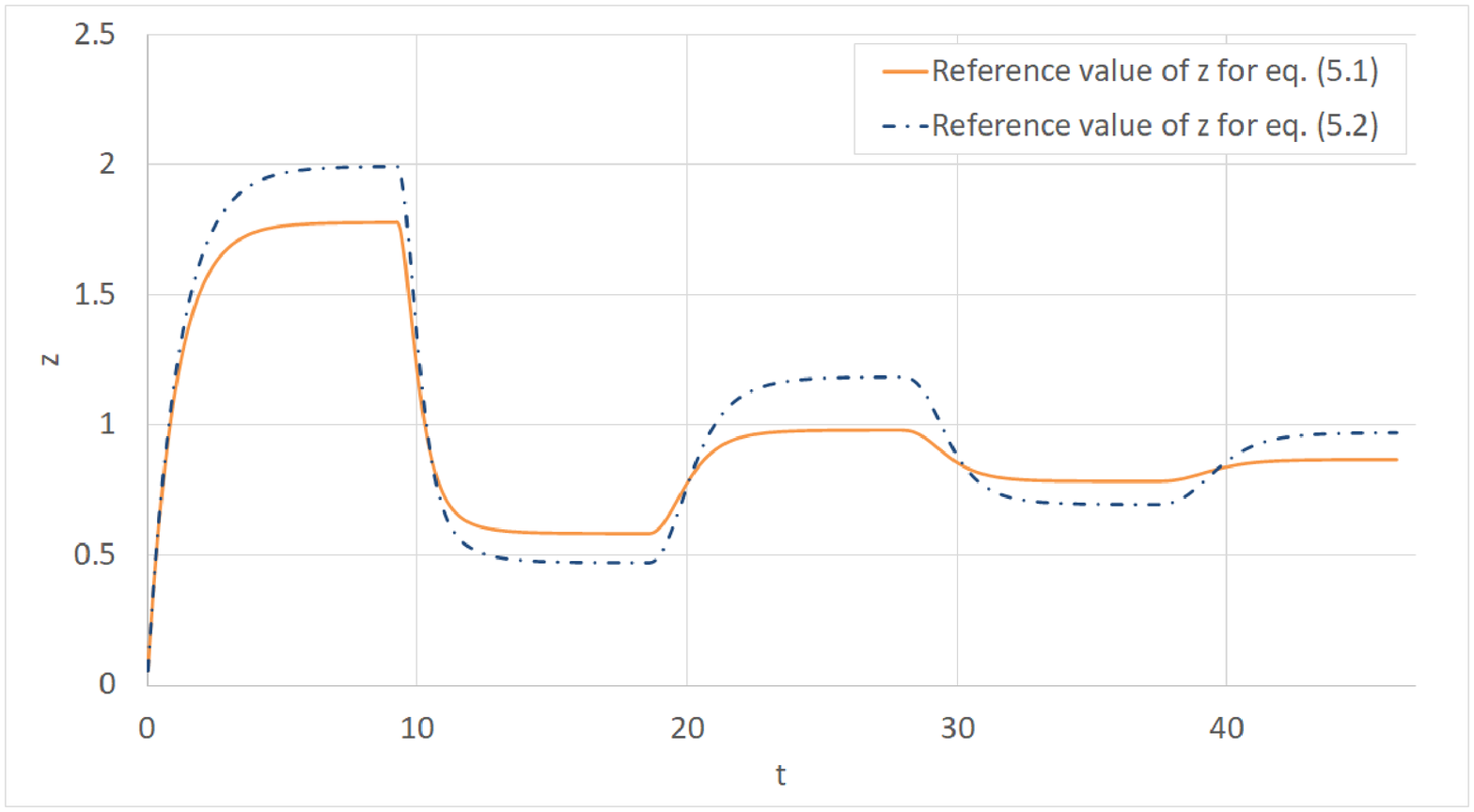}{0.8}{An approximated solution computed on the dense mesh treated as a reference solution of the metal phase change model \eqref{example_metal} and \eqref{example_metal_2}.}{pic:sol_test_eq}

\pic{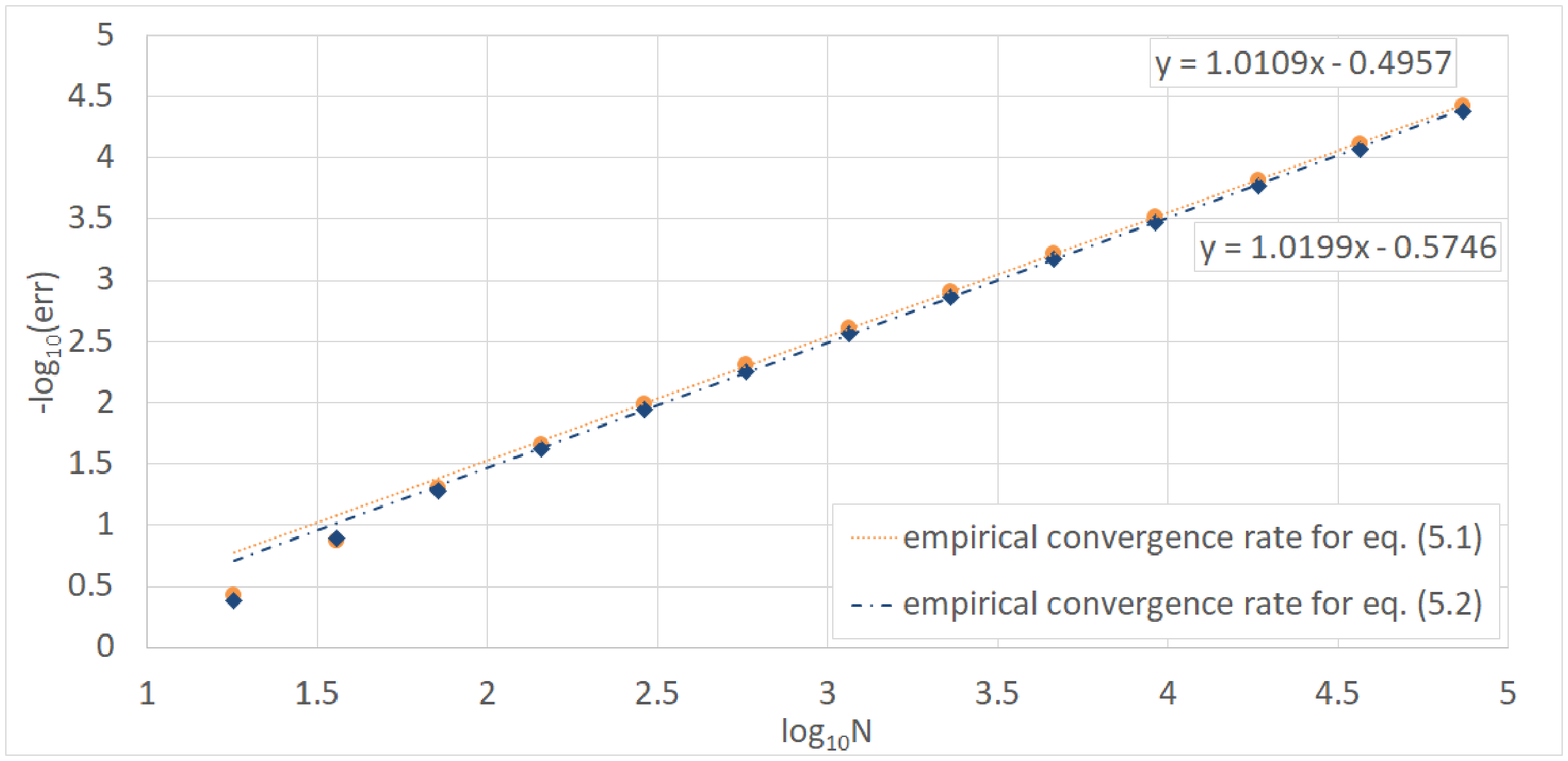}{0.8}{$-\log_{10} \textrm{(err)}$ vs. $\log_{10}N$ for the metal phase change model \eqref{example_metal} and \eqref{example_metal_2}.}{pic:conv_rate_test_eq}

\begin{figure}[!htbp]
	\centering
	\begin{subfigure}[t]{0.4\textwidth}
		\centering
		\includegraphics[height=3.5cm]{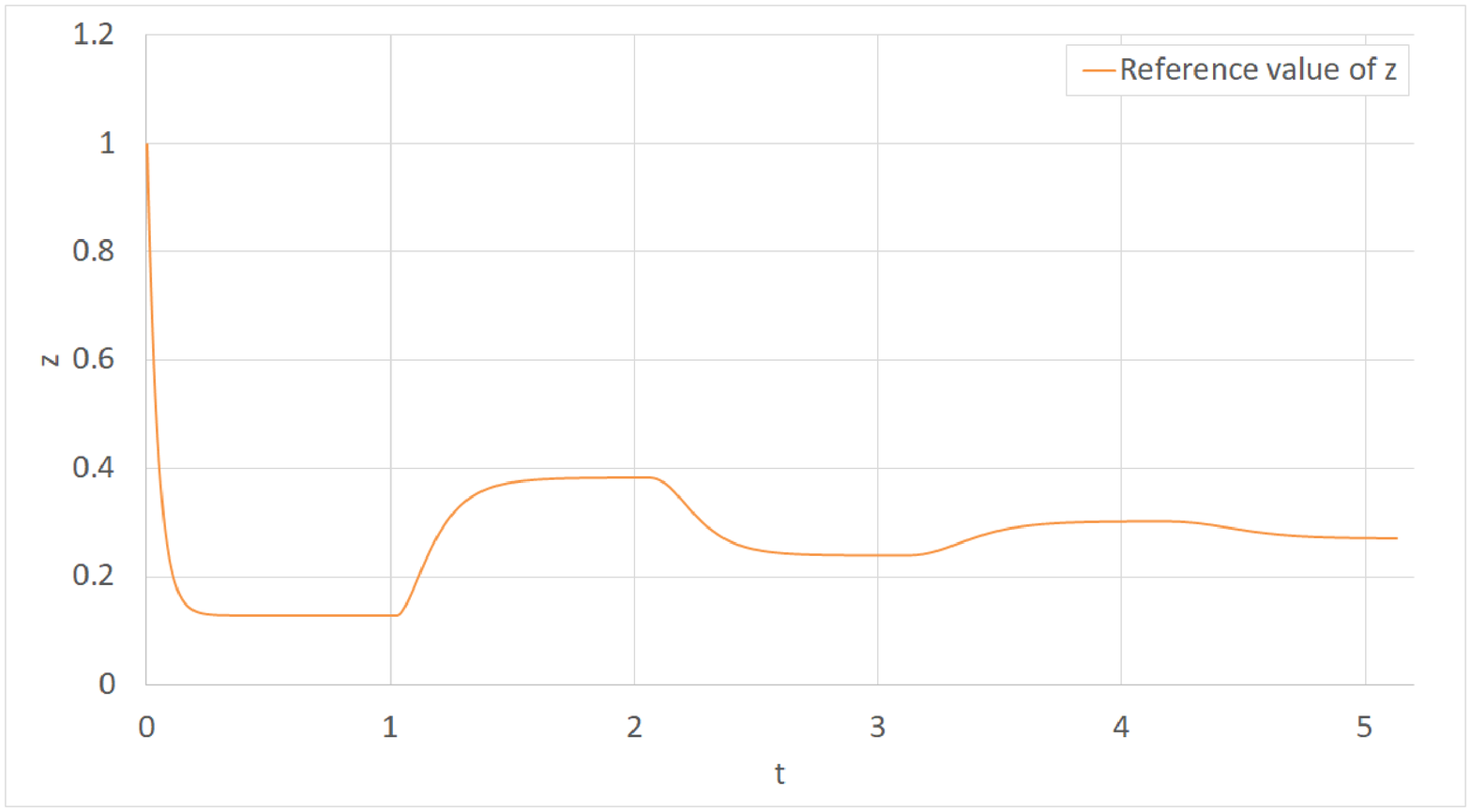}
		\caption{An approximated solution of \eqref{example_metal} for \linebreak $A=3.27$, $B=5.62$, $C=9.89$, $D=-7.31$, $\varrho=0.88$, $\gamma=0.89$, $\tau=1.03$, $z_0=1$, $t_0=0$, $n=5$.}\label{pic:sol_5.3_1}		
	\end{subfigure}
	\quad \quad
	\begin{subfigure}[t]{0.4\textwidth}
		\centering
		\includegraphics[height=3.5cm]{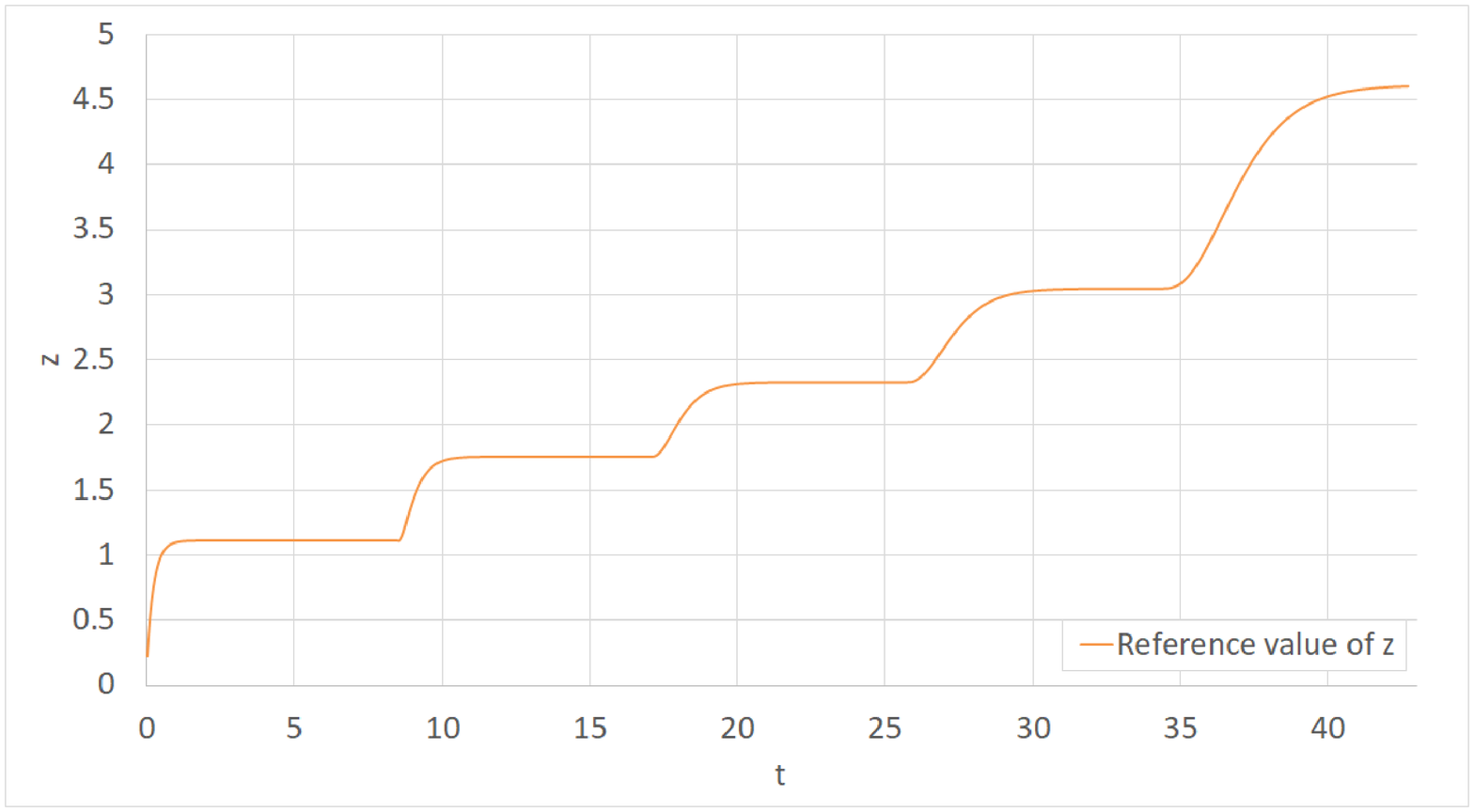}
		\caption{An approximated solution of \eqref{example_metal} for $A=5$, $B=6.62$, $C=0.52$, $D=4$, $\varrho=0.32$, $\gamma=0.33$, $\tau=8.55$, $z_0=0.22$, $t_0=0$, $n=5$.}\label{pic:sol_5.3_2}
	\end{subfigure}		
	\begin{subfigure}[t]{0.4\textwidth}
		\centering
		\includegraphics[height=3.5cm]{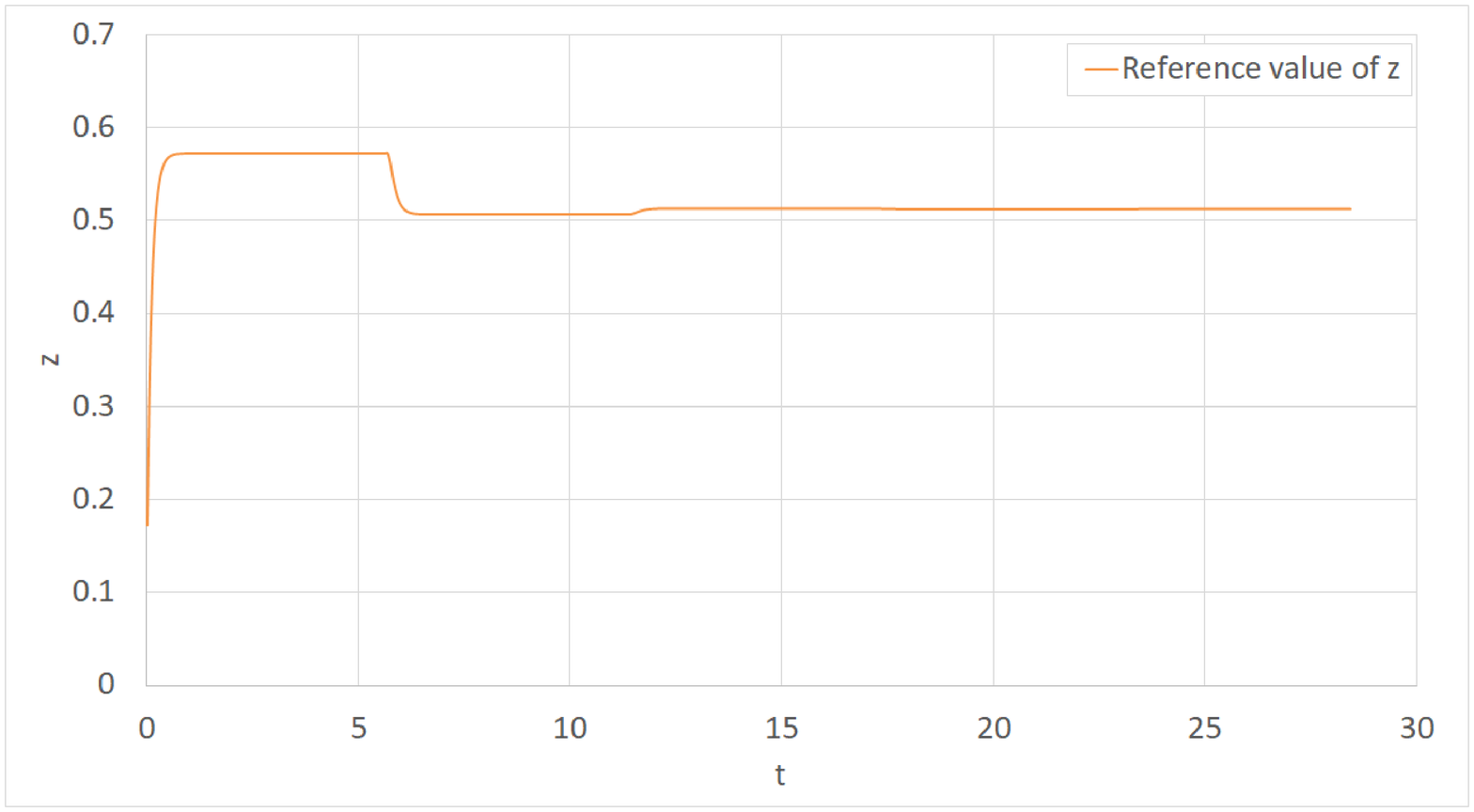}
		\caption{An approximated solution of \eqref{example_metal} for \linebreak $A=5.16$, $B=0.42$, $C=3.61$, \linebreak $D=-6.74$, $\varrho=0.99$, $\gamma=0.11$, $\tau=5.69$, $z_0=0.17$, $t_0=0$, $n=5$.}\label{pic:sol_5.3_3}		
	\end{subfigure}
	\quad
	\begin{subfigure}[t]{0.4\textwidth}
		\centering
		\includegraphics[height=3.5cm]{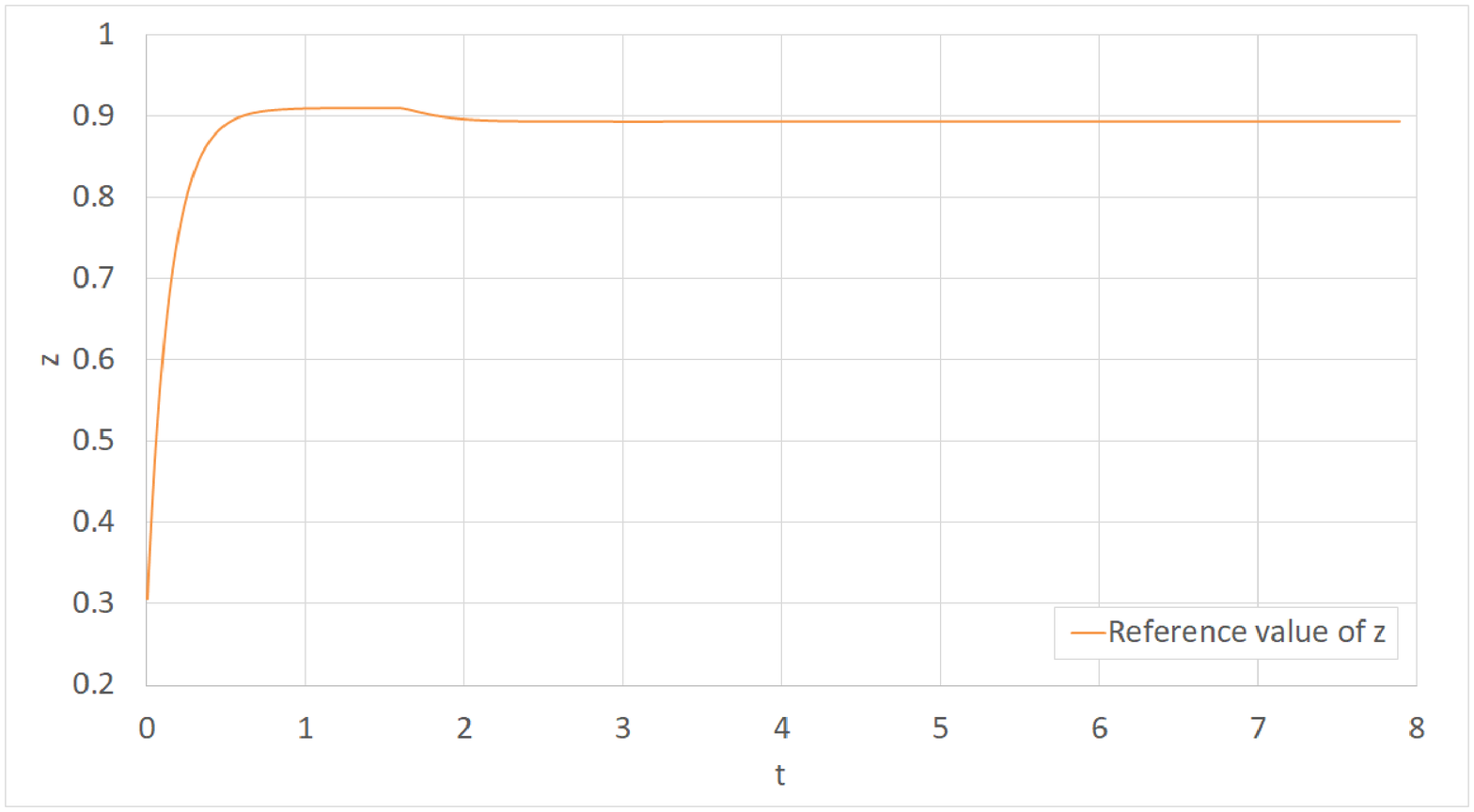}
		\caption{An approximated solution of \eqref{example_metal} for \linebreak $A=6.75$, $B=2.79$, $C=4.7$, \linebreak $D=-0.01$, $\varrho=0.86$, $\gamma=0.02$, $\tau=1.58$, $z_0=0.31$, $t_0=0$, $n=5$. }\label{pic:sol_5.3_4}
	\end{subfigure}
	\caption{Other exemplary reference solutions of the metal phase change model \eqref{example_metal}}\label{pic:sol_test_eq_5.3_other}
\end{figure}

We also consider a similar case with subtle changes, i.e. 
\begin{equation}
\label{example_metal_2}
    f(t,y,z)=A-B\cdot  \sgn(y)\cdot |y|-C\cdot \sgn(y)\cdot |y|^{\varrho}\cdot |z|+D \cdot y \cdot z,
\end{equation}
where $\varrho, \gamma \in (0,1]$, $A,B,C > 0$, and $D \in \R$. This function \eqref{example_metal_2} also satisfies assumptions (F1)-(F4), see Fact \ref{metal_assumptions_2}. For the numerical experiment we take the same values of parameters $A, B, C, D, \varrho, \gamma, \tau, z_0, t_0, n$ like in the \eqref{example_metal} case.
The approximated solution of \eqref{example_metal_2} is presented in Figure~\ref{pic:sol_test_eq} and the test results for this case in Figure~\ref{pic:conv_rate_test_eq}. As before the computation time behaves linear with respect to the number of computation points and for $N \cdot n = 90$ equals $0.002 s$ and for $N \cdot n = 368640$ equals $5.3525 s$.

\vspace{10pt}
\subsection{Model of releasing mature cells into the blood stream}
Another example is modeling the release of mature cells into the blood stream, so called Mackey–Glass equation (see Example 1.1.7, page 7 in \cite{Bellen} and originally introduced in \cite{Mackey}),
\begin{equation}
\label{example_blood}
    z'(t)=\frac{bz(t-\tau)}{1+[z(t-\tau)]^m} - az(t)
\end{equation}
with parameters $a=0.1$, $b=0.2$, $m=10$, $\tau=20$, and $n=500$. 
One can see that the equation \eqref{example_blood} does not fulfill the assumptions (F1)-(F4), still we present the example as a classical test problem where the method behaves properly. The solution of \eqref{example_blood} is presented in Figure~\ref{pic:sol_blood} and the test results for this case in Figure~\ref{pic:conv_rate_blood}. As before the computation time is linear with respect to the number of computation points and for $N \cdot n = 900$ equals $0.01793 s$ and for $N \cdot n = 3686400$ equals $25.6793 s$.

\pic{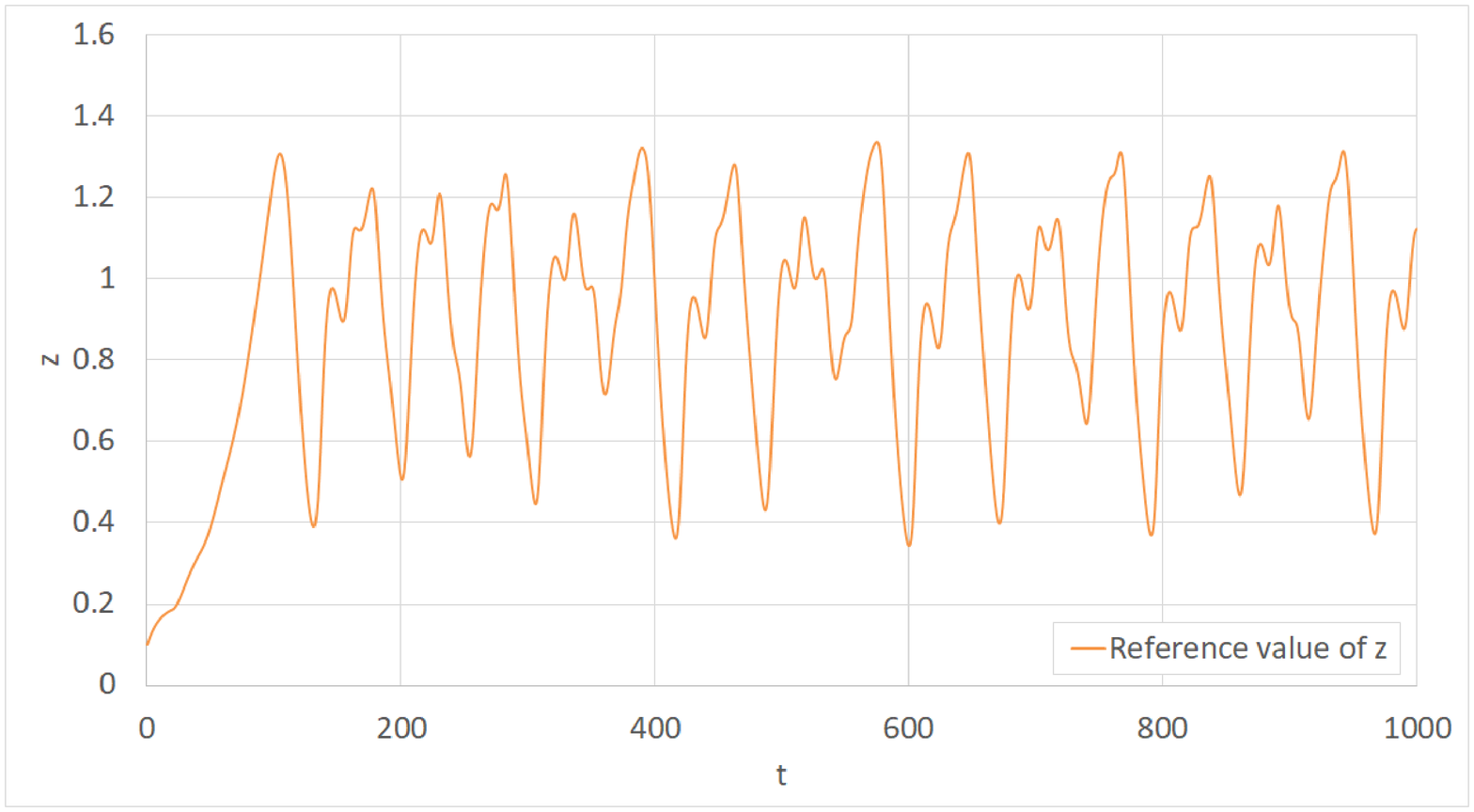}{0.8}{An approximated solution computed on the dense mesh, treated as the reference solution of the model of releasing mature cells into the blood stream \eqref{example_blood}.}{pic:sol_blood}

\pic{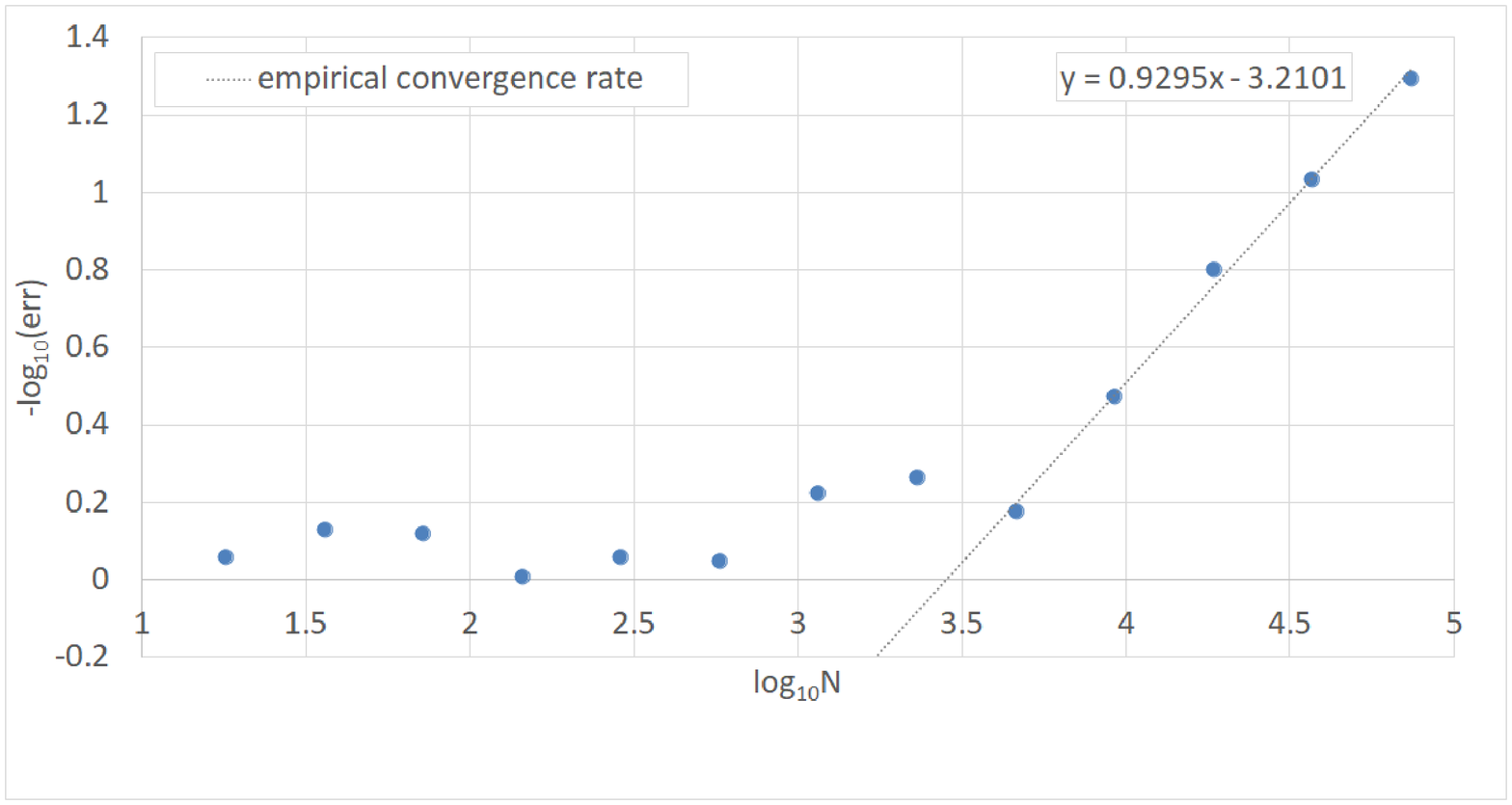}{0.8}{$-\log_{10} \textrm{(err)}$ vs. $\log_{10}N$ for the model of releasing mature cells into the blood stream \eqref{example_blood}.}{pic:conv_rate_blood}

\vspace{10pt}
\subsection{Epidemiology model}
Widely known usage of delay differential equations is a Susceptible-Infectious-Recovered (SIR) model and its modifications. This model family describes spread of the disease. By \cite{Mahrouf}, we consider the following multidimensional delay differential equation
\begin{equation}
\label{example_sir}
    \left\{ \begin{array}{lll}
        S'(t) &=& -\beta(1-u)\frac{S(t) I_s(t)}{N_{pop}}\\
        I_s'(t) &=& \beta\epsilon(1-u)\frac{S(t-\tau_1)I_s(t-\tau_1)}{N_{pop}} - \alpha I_s(t) - (1-\alpha)(\mu_s + \eta_s)I_s(t)\\
        I_a'(t) &=& \beta(1-\epsilon)(1-u)\frac{S(t-\tau_1)I_s(t-\tau_1)}{N_{pop}} - \eta_a I_a(t)\\
        F_b'(t) &=&\alpha \gamma_b I_s(t-\tau_2) - (\mu_b + r_b) F_b(t)\\
        F_g'(t) &=&\alpha \gamma_g I_s(t-\tau_2) - (\mu_g + r_g) F_g(t)\\
        F_c'(t) &=&\alpha \gamma_c I_s(t-\tau_2) - (\mu_c + r_c) F_c(t)\\
        R'(t) &=& \eta_s(1-\alpha)I_s(t-\tau_3) + \eta_a I_a(t-\tau_3) + r_b F_b(t-\tau_4) + r_g F_g(t-\tau_4) + r_c F_c(t-\tau_4)\\
        M'(t) &=& \mu_s(1-\alpha)I_s(t-\tau_3) + \mu_b F_b (t-\tau_4) + \mu_g F_g (t-\tau_4) + \mu_c F_c (t-\tau_4)
    \end{array} \right.
\end{equation}
with values of parameters $\beta=0.4517$, $\epsilon=0.794$,$\gamma_b = 0.8$, $\gamma_g = 0.15$, $\gamma_c = 0.05$, $\alpha = 0.06$, $\eta_a = \frac{1}{21}$, $\eta_s = \frac{0.8}{21}$, $\mu_s = \frac{0.01}{21}$, $\mu_b = 0$, $\mu_g = 0$, $\mu_c = \frac{0.4}{13.5}$, $r_b = \frac{1}{13.5}$, $r_g = \frac{1}{13.5}$, $r_c = \frac{0.6}{13.5}$, $\tau_1=5.5$, $\tau_2=7.5$, $\tau_3=21$, $\tau_4=13.5$, $N_{pop} = 35280000$, and 
\begin{displaymath}
    u = \left\{ \begin{array}{ll}
        0.2, & t\in (0,8],\\
        0.3, & t\in (8,18],\\
        0.4, & t\in (18,35],\\
        0.8, & t\in (35,240].\\
    \end{array} \right.
\end{displaymath}
As initial value vector we take $\eta=(35280000, 20, 0, 0, 0, 0, 0, 0)$.

It can be noticed that in \eqref{example_sir} we have to deal with 4 different values of delay. In order to compute the approximation we take one common delay $\tau=0.5$ and we compute a solution with that value of time lag for $n=480$ (which is equivalent to 240 days). Simultaneously we had to slightly change an algorithm to take proper delay values as arguments in the right-hand side function. A solution of SIR model can be analyzed in Figure~\ref{pic:sol_sir} (without presenting values of $S$ because of the big difference in order of magnitude) and the convergence rate is presented in Figure~\ref{pic:conv_rate_sir}. The computation time is linear with respect to the number of computation points and for $N \cdot n = 8640$ equals $0.586 s$ and for $N \cdot n = 4423680$ equals $85.902 s$.

\pic{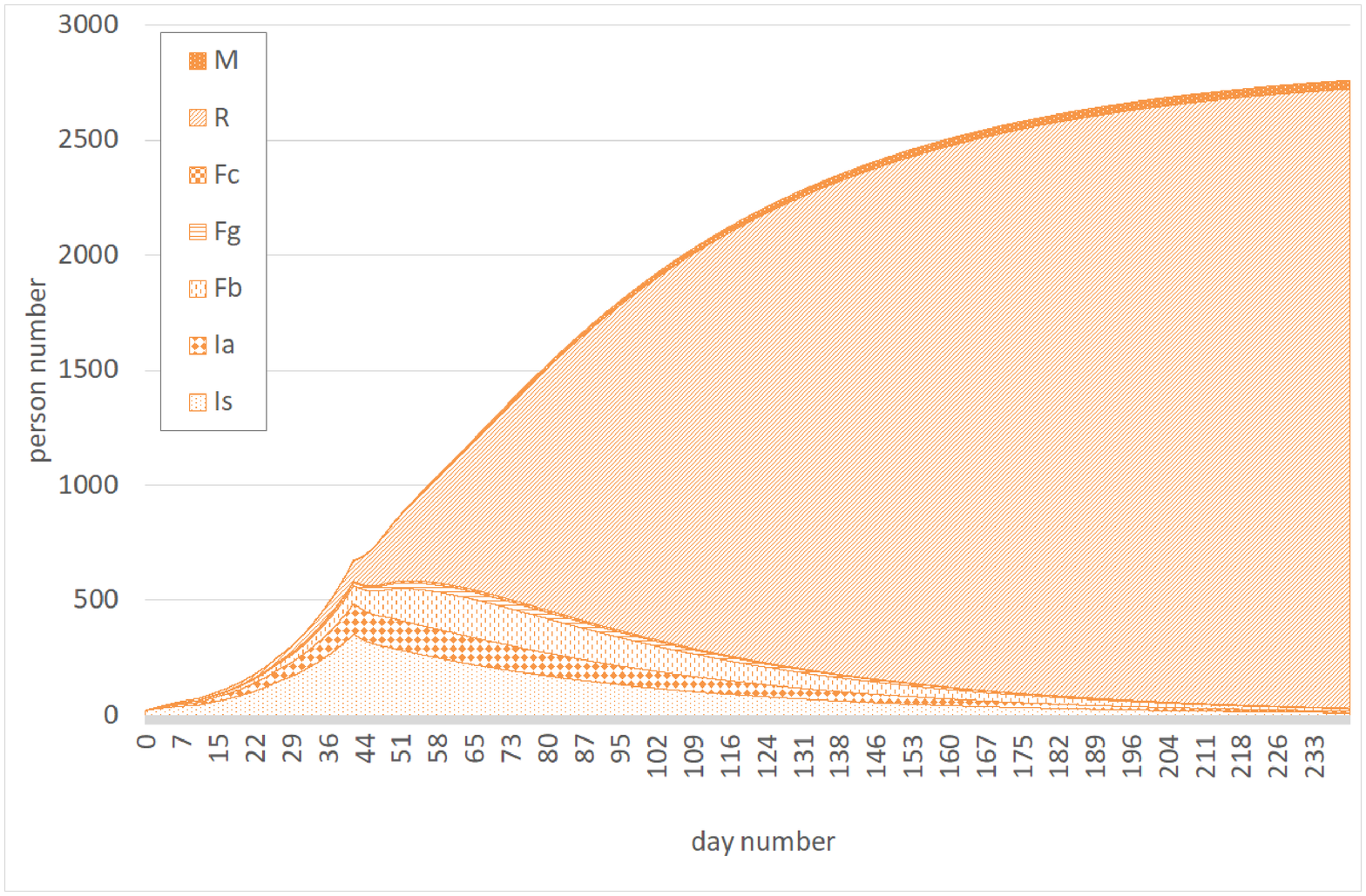}{0.8}{An approximated solution computed on the dense mesh treated as a reference solution of the SIR model \eqref{example_sir}, values of $S$ can be computed as a difference between $N_{pop}$ and values presented on the figure.}{pic:sol_sir}

\pic{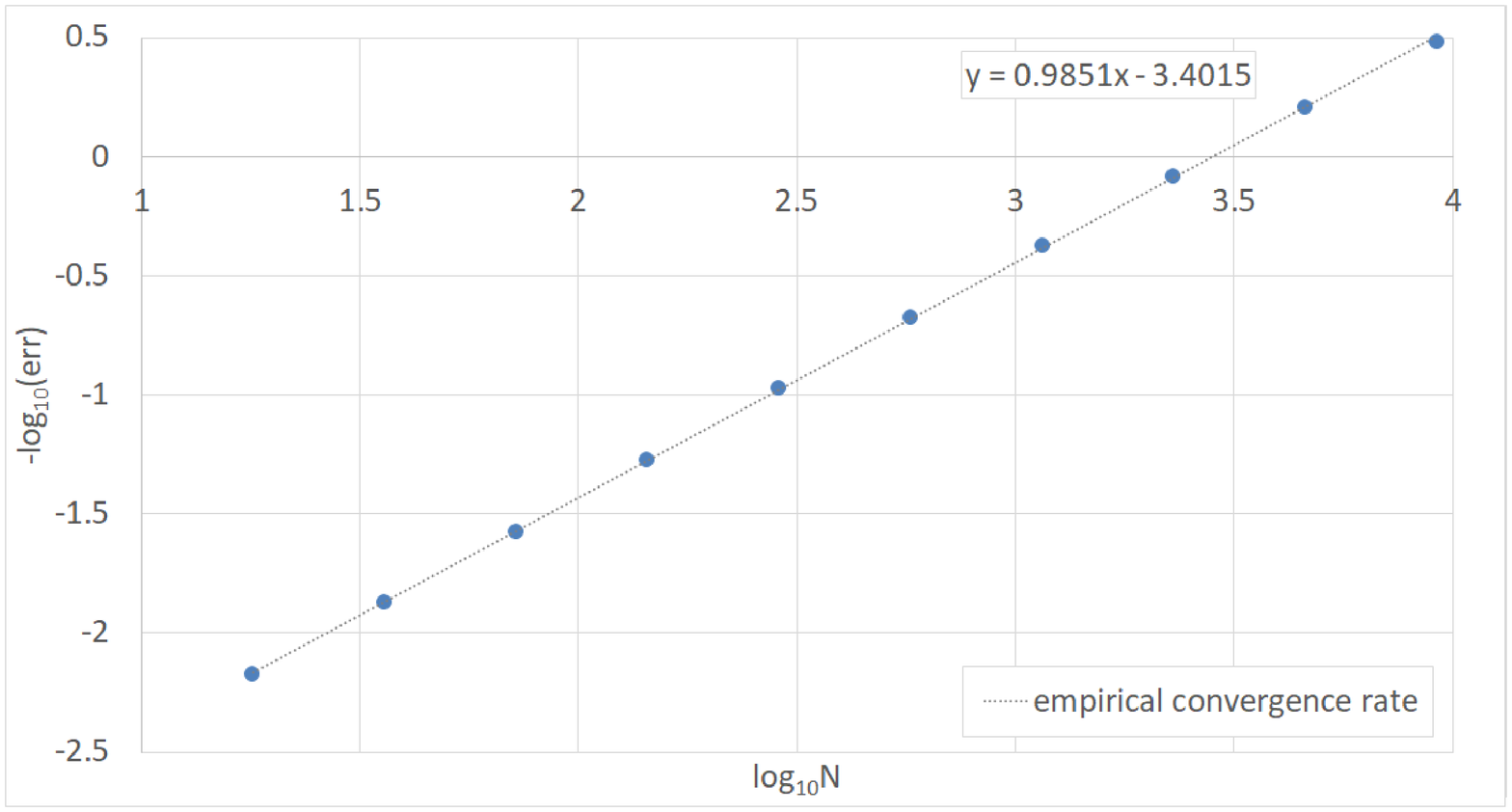}{0.8}{$-\log_{10} \textrm{(err)}$ vs. $\log_{10}N$ for the SIR model \eqref{example_sir}.}{pic:conv_rate_sir}
\section{Appendix - analytical properties of solutions of ODEs and its Euler approximation}
This section consist of some auxiliary results for solutions of ordinary differential equations and its Euler approximation that are used in the paper, especially for proving the main Theorem~\ref{rate_of_conv_expl_Eul}.
\begin{lem}
	\label{odes_exist_sol}
	Let us consider the following ODE
	\begin{equation}
	\label{ODE_1_Peano}
	    z'(t)=g(t,z(t)), \quad t\in [a,b], \quad z(a)=\xi,
	\end{equation}
	where $-\infty<a<b<+\infty$, $\xi\in\R^d$ and $g:[a,b]\times\R^d\to\R^d$ satisfies the following conditions
	\begin{itemize}
		\item [(G1)] $g\in C([a,b]\times\R^d;\R^d)$.
		\item [(G2)] There exists $K\in (0,+\infty)$ such that for all $(t,y)\in [a,b]\times\R^d$
		\begin{displaymath}
		\|g(t,y)\|\leq K(1+\|y\|).
		\end{displaymath}
		\item [(G3)] There exists $H \in \R$ such that for all $t\in [a,b]$, $y_1,y_2\in\R^d$
		\begin{displaymath}
		\langle y_1-y_2, g(t,y_1)-g(t,y_2)\rangle \leq H \| y_1 - y_2 \|^2.
		\end{displaymath}
	\end{itemize}	
	Then we have what follows.
	\begin{itemize}
	    \item [(i)] The equation \eqref{ODE_1_Peano} has a unique  solution $z\in C^1([a,b];\mathbb{R}^d)$ such that
	\begin{equation}
	\label{est_sol_z}
	\sup\limits_{t\in[a,b]}\|z(t)\|\leq (\|\xi\|+K(b-a))e^{K(b-a)},
	\end{equation}
	and for all $t,s\in [a,b]$
	\begin{equation}
	\label{lip_sol_z}
	\|z(t)-z(s)\|\leq \bar K |t-s|,
	\end{equation}
	where $\bar K=K\Bigl(1+(\|\xi\|+K(b-a))e^{K(b-a)}\Bigr)$.
	\item [(ii)] Let us consider $u:[a,b]\to\mathbb{R}^d$ the solution of 
	\begin{equation}
	\label{ODE_2_Peano}
	    u'(t)=g(t,z(t)), \quad t\in [a,b], \quad u(a)=\zeta,
	\end{equation}
	with $\zeta\in\mathbb{R}^d$. Then for all $t\in [a,b]$ we have
	\begin{equation}
	\label{dist_zu}
	    \|u(t)-z(t)\|\leq e^{H_{+}(b-a)}\|\xi-\zeta\|.
	\end{equation}
\end{itemize}
\end{lem}
\begin{proof} Since the right-hand side function $g$ is continuous and it is of at most linear growth, Peano's theorem guarantees existence of the $C^1$-solution (e.g. see Theorem 70.4, page 292 in \cite{gorniewicz}). Now, we show that the uniqueness follows from the one-sided Lipschitz assumption (G3). Namely, let us assume that \eqref{ODE_1_Peano} has two solutions $z=z(t)$ and $x=x(t)$ with the same initial-value $z(a)=x(a)=\xi$. By (G3) we have for all $t\in [a,b]$ that
	\begin{eqnarray}
	\label{diff_sol}
	&&\frac{d}{dt}\| z(t)-x(t)\|^2 = 2 \left\langle z(t)-x(t), g(t, z(t))-g(t, x(t)) \right\rangle\notag\\
	&&\leq 2 H \| z(t)-x(t) \|^2\leq 2 H_+ \| z(t)-x(t) \|^2.
	\end{eqnarray}
	Let us consider the $C^1$-function $[a,b]\ni t\to \varphi(t) =  \| z(t)-x(t)\|^2\in [0,+\infty)$, where $\varphi(a)=0$.
	Integrating two sides of the preceding inequality we get
	\begin{displaymath}
	\varphi(t)  \leq 2 H_+ \int_a^t \varphi(s) ds.
	\end{displaymath}
	Hence, from the Grownall's lemma we obtain that $\varphi(t)=0$ for all $t\in [a,b]$, which in turns implies that $z(t)=x(t)$ for all $t\in [a,b]$.
 
	Note that, by the assumption (G2), for all $t\in [a,b]$ it holds
	\begin{equation}
	\|z(t)\| \leq \|\xi\| + \int_a^t \|g(s, z(s))\| ds\leq  \|\xi\| + K(b-a) + K \int_a^t\| z(s) \| ds.
	\end{equation}
	Again by the Gronwall's lemma we obtain for all $t\in [a,b]$ that
	\begin{equation}
	\|z(t)\| \leq (\|\xi\| + K(b-a)) e^{K(t-a)}.
	\end{equation}
	This implies \eqref{est_sol_z}.
	
	By (G2) and \eqref{est_sol_z} we obtain for all $t, s \in [a,b]$
	\begin{eqnarray}
	&&\|z(t) - z(s)\| = \|z(t \vee s) - z(t \wedge s)\|  \leq \int_{t \wedge s}^{t \vee s} \|g(u, z(u))\| du\notag\\
	&&\leq K (1 + \sup_{t \in [a,b]} \|z(t)\|)(t \wedge s - t \vee s)
	\leq \bar{K} |t-s|,
	\end{eqnarray}
	where $\bar{K}=  K\Bigl(1+(\|\xi\|+K(b-a))e^{K(b-a)}\Bigr)$. 
	Hence, the proof of \eqref{lip_sol_z} is completed.
	
	As in \eqref{diff_sol} we get that for all $t\in [a,b]$
	\begin{equation}
	   \frac{d}{dt}\| z(t)-u(t)\|^2 \leq 2H_+\|z(t)-u(t)\|^2
	\end{equation}
	which implies that
	\begin{equation}
	    \|z(t)-u(t)\|^2\leq \|\xi-\zeta\|^2+2H_+\int\limits_a^t\|z(s)-u(s)\|^2ds.
	\end{equation}
	By using Gronwall's lemma we get \eqref{dist_zu}.
\end{proof}
\begin{lem}
	\label{eul_odes}
	Let us consider the  following ordinary differential equation
	\begin{equation}
	\label{ODE_1}
	z'(t)=g(t,z(t)), \quad t\in [a,b], \quad z(a)=\xi,
	\end{equation}
	where $-\infty<a<b<+\infty$, $\eta\in\R^d$ and 	$g:[a,b]\times\R^d\to\R^d$ satisfies the following conditions:
	\begin{itemize}
		\item [(G1)] $g\in C([a,b]\times\R^d;\R^d)$.
		\item [(G2)] There exists $K\in (0,+\infty)$ such that for all $(t,y)\in [a,b]\times\R^d$
		\begin{displaymath}
		\|g(t,y)\|\leq K(1+\|y\|).
		\end{displaymath}
		\item [(G3)] There exists $H \in \R$ such that for all $t\in [a,b]$, $y_1,y_2\in\R^d$
		\begin{displaymath}
		\langle y_1-y_2, g(t,y_1)-g(t,y_2)\rangle \leq H \| y_1 - y_2 \|^2.
		\end{displaymath}
		\item [(G4)] There exist $L\geq 0$, $p \in\mathbb{N}$, $\alpha$, $\beta_1,\beta_2,\ldots,\beta_p \in (0,1]$, such that for all $t_1,t_2\in [a,b]$, $y_1,y_2\in\R^d$
			\begin{displaymath}
			\|g(t_1,y_1)-g(t_2,y_2)\|\leq L\Bigl((1+\|y_1\|+\|y_2\|)\cdot |t_1-t_2|^{\alpha}+ \sum_{i=1}^p \|y_1-y_2\|^{\beta_i}\Bigr).
			\end{displaymath}
	\end{itemize}
	Let us consider the Euler method based on equidistant discretization. Namely, for $n\in\N$, $\Delta\in [0,+\infty)$ we set $h=(b-a)/n$, $t_k=a+kh$, $k=0,1,\ldots,n$, and let $y_0 \in \R^d$ be any vector from the ball $B(\xi, \Delta) = \{ y \in R^d: \|\xi-y\|\leq\Delta \}$. We take
	\begin{equation}
		\label{def_Euler_g}
			y_{k+1}=y_k+h \cdot g(t_k,y_k), \quad k=0,1,\ldots,n-1.
	\end{equation}
	Then the following holds
	\begin{itemize}
		\item [(i)] There exists $\tilde{C_1}=\tilde{C_1}(b-a,K)\in (0,+\infty)$ such that for all $n\in\N$, \linebreak $\Delta \in [0,+\infty)$, $\xi\in\mathbb{R}^d$ $y_0 \in B(\xi, \Delta)$  we have 
		\begin{equation}
		\label{eq:711}
			\max\limits_{0\leq k\leq n}\|y_k\|\leq \tilde{C_1}(1+\|\xi\|)(1+\Delta).
		\end{equation}	
		\item [(ii)] There exists $\tilde{C_2}=\tilde{C_2}(b-a,L,K,H,p,\alpha, \beta_1,\ldots,\beta_p)\in (0,+\infty)$ such that for all \linebreak $n\in\N$, $\Delta \in [0,+\infty)$,  $\xi\in\mathbb{R}^d$ $y_0 \in B(\xi, \Delta)$  we have 
		\begin{equation}
		\label{est_euler_err}
		\max\limits_{0\leq k\leq n}\|z(t_k)-y_k\|\leq \tilde{C_2}(1+\|\xi\|)\Bigl(\Delta+h^\alpha +  \sum\limits_{i=1}^p h^{\beta_i}\Bigr).
		\end{equation}
	\end{itemize}
\end{lem}
\begin{proof} Fix $n\in\N$, $\Delta\in [0,+\infty)$, $\xi\in\mathbb{R}^d$ and $y_0\in B(\xi,\Delta)$.	By the assumption (G2) we have that for all $k=0,1,\ldots,n-1$
	\begin{displaymath}
		\|y_{k+1}\|\leq \|y_k\|+h \|g(t_k,y_k)\| \leq (1+hK)\|y_k\|+hK
	\end{displaymath}
	and, since $y_0 \in B(\xi, \Delta)$, $\|y_0\|\leq  \Delta + \|\xi\|$.
	Hence, by the discrete version of Gronwall's lemma we get that for all $k=0,1,\ldots,n$ that
	\begin{displaymath}
		\label{EST_yk_gronwall}
		\|y_k\|\leq e^{K(b-a)}(\Delta + \|\xi\|) + e^{K(b-a)}-1\leq e^{K(b-a)}(1+\|\xi\|)(1+\Delta).
	\end{displaymath}
This proves \eqref{eq:711} with $\tilde C_1= e^{K(b-a)}$.
	
	We now prove \eqref{est_euler_err}. For  $k=0,1,\ldots,n-1$ we consider the following local ODE
	\begin{equation}
		\label{def_local_zk}
		z_k'(t)=g(t,z_k(t)), \quad t\in [t_k,t_{k+1}], \quad z_k(t_k)=y_k.
	\end{equation}
	By Lemma~\ref{odes_exist_sol} there exists unique solution $z_k:[t_k,t_{k+1}]\to\mathbb{R}^d$ of \eqref{def_local_zk}. From the assumption (G2) and by \eqref{eq:711} we get for all $t\in [t_k,t_{k+1}]$, $k=0,1,\ldots,n-1$ that
	\begin{displaymath}
		\|z_k(t)\|
		\leq \tilde{C_1} (1+\|\xi\|)(1+\Delta)+K(b-a)+K\int\limits_{t_k}^t \|z_k(s)\|ds.
	\end{displaymath}
	The use of Gronwall's lemma yields
	\begin{equation}
	\label{est_sup_zk}
		\max\limits_{0\leq k\leq n-1}\sup\limits_{t\in [t_k,t_{k+1}]}\|z_k(t)\|\leq C_2 (1+\|\xi\|)(1+\Delta),
	\end{equation}
	where $C_2=e^{K(b-a)} \Bigl( \tilde{C_1}+K(b-a) \Bigr)$. By (G2) and \eqref{est_sup_zk} we get  for all $t\in [t_k,t_{k+1}]$, \linebreak $k=0,1,\ldots,n-1$
	\begin{eqnarray}
	\label{est_sup_zkyk}
		&&\|z_k(t)-y_k\|\leq \int\limits_{t_k}^t \|g(s,z_k(s))\|ds\leq hK\Bigl(1+\sup\limits_{t\in [t_k,t_{k+1}]}\|z_k(t)\|\Bigr)\notag\\
		&&\leq hK\Bigl(1+C_2(1+\|\xi\|)(1+\Delta)\Bigr)\leq hC_3(1+\|\xi\|)(1+\Delta),
	\end{eqnarray}
	with $C_3=K(1+C_2)$. From Lemma \ref{odes_exist_sol} (ii) we arrive at
	\begin{eqnarray}
		\label{EST_INEQ1}
		\|z(t_{k+1})-y_{k+1}\|&\leq& \|z(t_{k+1})-z_k(t_{k+1})\|+\|z_k(t_{k+1})-y_{k+1}\|\notag\\
		&\leq& e^{hH_+}\|z(t_k)-y_k\|+\|z_k(t_{k+1})-y_{k+1}\|
	\end{eqnarray}
	for $k=0,1,\ldots,n-1$.  Using (G4), \eqref{eq:711}, \eqref{est_sup_zk} and \eqref{est_sup_zkyk} we get
	\begin{eqnarray}
	&&\|z_k(t_{k+1})-y_{k+1}\|\leq\int\limits_{t_k}^{t_{k+1}}\|g(s,z_k(s))-g(t_k,y_k)\|ds \notag\\
	&& \leq L\int\limits_{t_k}^{t_{k+1}}\Bigl((1+\|z_k(s)\|+\|y_k\|)\cdot (s-t_k)^{\alpha}+ \sum_{i=1}^p \|z_k(s)-y_k\|^{\beta_i}\Bigr)ds\notag \\
	\label{EST_INEQ21}
	&&\leq Lh\Bigl((1+(\tilde{C_1}+C_2)(1+\|\xi\|)(1+\Delta))\frac{1}{\alpha+1}\cdot h^{\alpha}+ \sum_{i=1}^p h^{\beta_i}C_3^{\beta_i}(1+\|\xi\|)^{\beta_i}(1+\Delta)^{\beta_i}\Bigr).\notag
	\end{eqnarray}
	It is easy to see that for all $x\in\R_+\cup\{0\}$, $\varrho\in(0,1]$
	\begin{equation}
	    \label{err_glob_ineq}
	    (1+x)^\varrho\leq  2(1+x).
	\end{equation}
	Hence
	\begin{equation}
	\label{EST_INEQ21_end}
	 \|z_k(t_{k+1})-y_{k+1}\|\leq C_4 (1+\|\xi\|)(1+\Delta) h \left( h^\alpha + \sum\limits_{i=1}^p h^{\beta_i}\right),
	\end{equation}
	where $\displaystyle{C_4=L\max\Bigl\{\frac{1+\tilde{C_1}+C_2}{\alpha+1}, 4 \max_{1\leq j \leq p}C_3^{\beta_j} \Bigr\}}$. From the above  considerations  we see that $C_4$ depends only on $\alpha, \beta_1,\ldots,\beta_p,p,L,K,b-a$. By \eqref{EST_INEQ1} and \eqref{EST_INEQ21_end} we get
    \begin{equation}
        \label{err_glob}
        \|z(t_{k+1})-y_{k+1}\|\leq e^{ hH_+ } \cdot \| z(t_k)-y_k\| +C_4 (1+\|\xi\|)(1+\Delta) h \left(h^\alpha + \sum\limits_{i=1}^p h^{\beta_i}\right).
    \end{equation}
	Let us denote
	\begin{equation}
	e_k=z(t_k)-y_k, \quad k=0,1,\ldots,n.
	\end{equation}
	Of course $\|e_0 \| = \| \xi -y_0 \| \leq \Delta$.  Hence, we arrive at the following recursive inequality  
	\begin{equation}
	\|e_{k+1}\|\leq e^{hH_{+}}\|e_k\| + C_4 (1+\|\xi\|)(1+\Delta) h \left(h^\alpha + \sum\limits_{i=1}^p h^{\beta_i}\right),
	\end{equation}
	for $k=0,1,\ldots,n-1$. Applying the Gronwall's lemma we get when $H_+ > 0$
	\begin{eqnarray}
	    \label{err_disc_gronwall_Hnon0}
	    &&\|e_k\| \leq e^{H_+(b-a)} \Delta + \frac{e^{H_+(b-a)}-1}{e^{hH_+}-1} C_4 (1+\|\xi\|)(1+\Delta) h \left(h^\alpha + \sum\limits_{i=1}^p h^{\beta_i}\right)\notag\\
	    &&\leq e^{H_+(b-a)} \Delta + \frac{e^{H_+(b-a)}-1}{H_+} C_4 (1+\|\xi\|)(1+\Delta) \left(h^\alpha + \sum\limits_{i=1}^p h^{\beta_i}\right)
	\end{eqnarray}
 for $k=0,1,\ldots,n$, and when $H_+=0$
	\begin{equation}
	    \label{err_disc_gronwall_H0}
	    \|e_k\| \leq \Delta +  C_4 (1+\|\xi\|)(1+\Delta)(b-a) \left(h^\alpha + \sum\limits_{i=1}^p h^{\beta_i}\right)
	\end{equation}
	for $k=0,1,\ldots,n$. By elementary calculations we arrive at \eqref{est_euler_err}.
\end{proof}

The following fact is well-known, see, for example, pages 3-4 in \cite{RF1}.

\begin{lem}
\label{aux_lem1}
 For all $\varrho\in (0,1]$ and $x,y\in\R$ it holds
	\begin{equation}
		|x|^{\varrho}\leq 1+|x|,
	\end{equation}
	and
	\begin{equation}
		\Bigl||x|^{\varrho}-|y|^{\varrho}\Bigl|\leq |x-y|^{\varrho}.
	\end{equation}
\end{lem}
Below we establish main properties of the functions \eqref{example_metal} and \eqref{example_metal_2}.
\begin{fact}
    \label{metal_assumptions_1}
    Let $A,B,C\geq 0$, $D\in\R$, $\varrho, \gamma\in (0,1]$ and consider $f:[0,+\infty)\times\R\times\R\to\R$ defined as follows\footnote{$\sgn(x)=1$ if $x\geq 0$ and $\sgn(x)=-1$ if $x<0$}
    \begin{displaymath}
        f(t,y,z)=A-B\cdot  \sgn (y)\cdot |y|-C\cdot \sgn(y)\cdot |y|^{\varrho}\cdot |z|^\gamma +D \cdot y \cdot |z|^\gamma.
    \end{displaymath}
    Then the function $f$ satisfies the assumptions (F1)-(F4).
\end{fact}
\begin{proof}
Let us define $h_1(y)=-\sgn(y)\cdot |y|=-y$, $h_2(y)=-\sgn(y)\cdot |y|^{\varrho}$ 
    for all $y\in\R$. Then we can write that
    \begin{equation}
	    f(t,y,z)=A+B\cdot h_1(y)+C\cdot h_2(y)\cdot |z|^\gamma + D \cdot y \cdot |z|^\gamma.
    \end{equation}
    Of course $h_1 \in C(\R)$. Moreover,
    \begin{equation}
        \lim\limits_{y\to 0-}h_2(y)=\lim\limits_{y\to 0-}(-y)^{\varrho}=0=h_2(0)=\lim\limits_{y\to 0+}h_2(y)=\lim\limits_{y\to 0+}(-1)\cdot y^{\varrho},
    \end{equation}
    therefore  $h_2 \in C(\R)$. In particular, this implies that  $f \in C([0, \infty) \times \R \times \R)$, so $f$ satisfies (F1).

By Lemma \ref{aux_lem1} we have $|h_1(y)| =  |y| \leq 1+|y|$, $|h_2(y)|  = |y|^{\varrho} \leq 1+|y|$  for all $y\in\R$.
    Hence, for $(t, y, z) \in [0,+\infty)\times\R\times\R$
    \begin{eqnarray*}
	    |f(t,y,z)| &\leq& A+B\cdot |h_1(y)|+C\cdot |h_2(y)|\cdot |z|^\gamma + |D||y||z|^\gamma \\
	    &\leq& A+B (1+|y|)+C (1+|y|)(1+|z|) + |D|(1+|y|)(1+|z|) \\
	    &\leq& (A+B+C+|D|)(1+|y|)(1+|z|),
    \end{eqnarray*}
and therefore $f$ satisfies (F2).    

For all $t\geq 0$,$z,y_1,y_2\in\mathbb{R}$
    \begin{eqnarray*}
        &&(y_1-y_2)\Bigl(f(t,y_1,z)-f(t,y_2,z)\Bigr) \\
        &=& B (y_1-y_2) \Bigl(h_1(y_1)-h_1(y_2)\Bigr) +C\cdot |z|^\gamma (y_1-y_2)\Bigl(h_2(y_1)-h_2(y_2)\Bigr) + D \cdot |z|^\gamma (y_1-y_2)^2
    \end{eqnarray*}
    Since $h_1, h_2$ are decreasing, it holds for all $y_1, y_2 \in\R$, $i=1,2$ that $\displaystyle{(y_1-y_2) (h_i(y_1)-h_i(y_2)) \leq 0}$. Moreover $B, C, |z|^\gamma \geq 0$, hence
    \begin{displaymath}
        (y_1-y_2)\Bigl(f(t,y_1,z)-f(t,y_2,z)\Bigr) \leq D \cdot |z|^\gamma (y_1-y_2)^2 \leq |D| \cdot (1+|z|) (y_1-y_2)^2,
    \end{displaymath}
    and $f$ satisfies (F3).

For all $y_1, y_2 \in\R$ we have that $|h_1(y_1)-h_1(y_2)| = |y_1 - y_2|$.  We now justify that  $h_2$ satisfies for all $y_1, y_2 \in\R$
    \begin{equation}
        |h_2(y_1)-h_2(y_2)| \leq 2|y_1 - y_2|^\varrho.
    \end{equation}
    When $y_1, y_2 < 0$ or $y_1, y_2 \geq 0$, by Lemma~\ref{aux_lem1}, we have 
    \begin{displaymath}
        |h_2(y_1)-h_2(y_2)| =  \big| |y_1|^{\varrho} - |y_2|^{\varrho} \big| \leq |y_1 - y_2|^{\varrho} \leq 2 |y_1 - y_2|^{\varrho}.
    \end{displaymath}
    For the case when $y_1<0$, $y_2\geq 0$ (the case $y_1\geq 0$, $y_2< 0$ is analogous) we have
    \begin{displaymath}
        |h_2(y_1)-h_2(y_2)| = 
        ||y_1|^{\varrho}+y_2^{\varrho}|=|y_1|^{\varrho}+y_2^{\varrho}\leq 2|y_1-y_2|^{\varrho},
    \end{displaymath}
     since $-y_1>0$,  $|y_1-y_2|=y_2+(-y_1)\geq y_2\geq 0$, $|y_1-y_2|=y_2+(-y_1)\geq -y_1=|y_1|\geq 0$,
    and $[0,+\infty)\ni x\to x^{\varrho}$ is increasing. Combining the facts above we obtain for all $t_1,t_2\geq 0$, $y_1,y_2,z_1,z_2\in\mathbb{R}$
    \begin{eqnarray*}	
	    &&|f(t_1,y_1,z_1)-f(t_2,y_2,z_2)| \\
	    &\leq& B|h_1(y_1)-h_1(y_2)|+C\Bigl||z_1|^\gamma\cdot h_2(y_1)-|z_2|^\gamma\cdot h_2(y_2)\Bigr| + |D|\cdot \Bigl|y_1 |z_1|^\gamma - y_2 |z_2|^\gamma\Bigr|\\
	    &\leq& B|y_1-y_2|+C\cdot |z_1|^\gamma\cdot |h_2(y_1)-h_2(y_2)|+C\cdot |h_2(y_2)|\cdot \Bigl||z_1|^\gamma-|z_2|^\gamma\Bigl| \\
	    &+& |D|\cdot |y_2| \cdot\Bigl| |z_1|^\gamma - |z_2|^\gamma\Bigr| + |D|\cdot |z_1|^\gamma \cdot |y_1 - y_2| \\
	    &\leq& B|y_1-y_2|+2 C (1+|z_1|) |y_1 - y_2|^\varrho +C(1+|y_2|)\cdot |z_1-z_2|^\gamma \\
	    &+& |D|\cdot |y_2|\cdot |z_1 - z_2|^\gamma + |D|\cdot (1+|z_1|)\cdot |y_1 - y_2| \\
	    &\leq& \max\{B+|D|, 2C, C+|D|\} \Big[ (1+|z_1|+|z_2|)(1+|y_1|+|y_2|)|t_1-t_2| \\ 
	    &+& (1+|z_1|+|z_2|)\cdot |y_1-y_2|+(1+|z_1|+|z_2|)\cdot |y_1-y_2|^{\varrho} + (1+|y_1|+|y_2|)\cdot |z_1-z_2|^\gamma \Big].
    \end{eqnarray*}
That ends the proof.
\end{proof}

The proof of the fact below is analogous to the proof of Fact~\ref{metal_assumptions_1} and is omitted.
\begin{fact}
    \label{metal_assumptions_2}
    Let $A,B,C\geq 0$, $D\in\R$, $\varrho, \gamma\in (0,1]$ and define a function $f:[0,+\infty)\times\R\times\R\to\R$ as follows
    \begin{displaymath}
        f(t,y,z)=A-B\cdot  \sgn (y)\cdot |y|-C\cdot \sgn(y)\cdot |y|^{\varrho}\cdot |z| +D \cdot y \cdot z.
    \end{displaymath}
    Then the function $f$ satisfies the assumptions (F1)-(F4).
\end{fact}
The proof of the following fact is straightforward.
\begin{fact}
\label{fakt_h}
For all $h \in (0,\frac{1}{2})$ it holds
\begin{displaymath}
 0 < \frac{1}{1-h} \leq 1 + 2h \leq 2.
\end{displaymath}
\end{fact}
 {\noindent\bf Statements and Declarations}
\newline
The authors declare that they have no conflict of interest

The datasets generated during and analysed during the current study are available from the corresponding author on reasonable request. 
\vspace{12pt} \newline
 {\noindent\bf Acknowledgments}
\newline
This research was partly supported by the National Science Centre, Poland, under project 2017/25/B/ST1/00945.

We would like to thank two anonymous referees for their valuable comments and suggestions that helped to improve the presentation of the results and quality of this paper.

\end{document}